\documentclass[]{amsart}
\usepackage{amssymb,amsmath, graphicx}
\usepackage{mathdots}
\usepackage{mathrsfs}
\usepackage{wrapfig}
\usepackage{diagrams}
\usepackage{enumerate}
\usepackage{url}
\usepackage[utf8x]{inputenc}

\def\qmod#1#2{{\hbox{}^{\displaystyle{#1}}}\!\big/\!\hbox{}_{
\displaystyle{#2}}}

\def\Qmod#1#2{{\hbox{}^{\displaystyle{#1}}}\!\bigg/\!\hbox{}_{
\displaystyle{#2}}}

\def\resto#1#2{{
#1\hskip 0.4ex\vline_{\hskip 0.2ex\raisebox{-0,2ex}
{{${\scriptstyle #2}$}}}}}

\def\map{\longrightarrow}
\def\textmap#1{\mathop{\vbox{\ialign{
                                  ##\crcr
      ${\scriptstyle\hfil\;\;#1\;\;\hfil}$\crcr
      \noalign{\kern 1pt\nointerlineskip}
      \rightarrowfill\crcr}}\;}}
\def\bigtextmap#1{\mathop{\vbox{\ialign{
                                  ##\crcr
      ${\hfil\;\;#1\;\;\hfil}$\crcr
      \noalign{\kern 1pt\nointerlineskip}
      \rightarrowfill\crcr}}\;}}
      
\def\textlmap#1{\mathop{\vbox{\ialign{
                                  ##\crcr
      ${\scriptstyle\hfil\;\;#1\;\;\hfil}$\crcr
      \noalign{\kern-1pt\nointerlineskip}
      \leftarrowfill\crcr}}\;}}

\def\C{{\mathbb C}}

\def\N{{\mathbb N}}

\def\P{{\mathbb P}}

\def\R{{\mathbb R}}
\def\Z{{\mathbb Z}}

\def\ug{{\mathfrak u}}

\def\Mg{{\mathfrak M}}

\newtheorem{sz}{Satz}[section]
\newtheorem{thry}[sz]{Theorem}
\newtheorem{pr}[sz]{Proposition}
\newtheorem{re}[sz]{Remark}
\newtheorem{co}[sz]{Corollary}
\newtheorem{dt}[sz]{Definition}
\newtheorem{lm}[sz]{Lemma}

\def\U{\mathrm{U}}

\def\Pic{\mathrm {Pic}}

\def\deg{\mathrm {deg}}
\def\Hom{\mathrm{Hom}}

\def\id{ \mathrm{id}}

\def\rk{\mathrm {rk}}

\def\Tors{\mathrm {Tors}}
\def\supp{\mathrm {supp}}
\def\BN{\mathrm {BN}}

\def\U2{\mathrm{U(2)}}
\def\niq{=\kern-.18cm /\kern.08cm}

\newcommand{\cal}{\mathcal}

\def\TE{{^{\cal T}}\hspace{-1.1mm}{\cal E}}
\def\TS{{^{\cal T}}\hspace{-1.1mm}{\cal S}}

\def\Ext{\mathrm{Ext}}
\def\Ann{\mathrm{Ann}}

\begin{document}

\title{On the torsion of the first direct image of a locally free sheaf}
\author{Andrei Teleman}
\thanks{The author has been partially supported by the ANR project MNGNK, decision 
Nr.  ANR-10-BLAN-0118}
\thanks{The author is indebted to the referee for detailed and valuable  comments and for helpful suggestions.}
\address{Aix Marseille Université
CNRS, Centrale Marseille, I2M, UMR 7373 13453 Marseille, France } \email{andrei.teleman@univ-amu.fr}

\begin{abstract}
Let $\pi:M\to B$ be a proper holomorphic submersion between complex manifolds and ${\cal E}$ a holomorphic bundle on $M$. We study and describe explicitly the torsion subsheaf  $\Tors(R^1\pi_*({\cal E}))$ of the first direct image $R^1\pi_*({\cal E})$ under the assumption $R^0\pi_*({\cal E})=0$. We give two applications   of our results. The first concerns the locus of points in the base of a generically versal family of complex surfaces where the family is non-versal.  The second application is a vanishing result for $H^0(\Tors(R^1\pi_*({\cal E})))$ in a concrete situation related to our program to prove the existence of curves on class VII surfaces.
\end{abstract}
  
\maketitle

\setcounter{section}{-1}

\section{Introduction}

Let $B$, $M$ be complex manifolds, $\pi:M\to B$ a proper, surjective holomorphic submersion, and let ${\cal E}$ be a holomorphic bundle on  $M$.  Denote by ${\cal E}_x$ the restriction of ${\cal E}$ to the fiber $M_x$, regarded as a locally free sheaf on $M_x$. Grauert's  locally freeness theorem states that a direct image $R^k\pi_*({\cal E})$ is locally free when the map $x\mapsto h^k({\cal E}_x)$ is constant on $B$.  The starting point of this article is the following natural questions:  What can be said about the singularities of the sheaf $R^k\pi_*({\cal E})$ when this condition is not satisfied? Can one describe explicitly the torsion subsheaf of  $R^k\pi_*({\cal E})$?

In this article we will deal with these questions for $k=1$ under the assumption $R^0\pi_*({\cal E})=0$.
  Our first result concerns  the support of the torsion  subsheaf   $\Tors(R^1\pi_*({\cal E}))$ and states

\newtheorem*{th-supp}{Theorem \ref{supp}} 
\begin{th-supp}
Let $B$, $M$ be complex manifolds, $\pi:M\to B$ a proper, surjective holomorphic submersion, and ${\cal E}$ a holomorphic bundle on  $M$. If $R^0\pi_*({\cal E})=0$ then the support  $\supp(\Tors(R^1\pi_*({\cal E})))$ coincides with the maximal pure 1-codimensional analytic subset of $B$ which is contained  in the Brill-Noether locus
 $$\BN_\pi({\cal E}):=\{x\in B |\ h^0({\cal E}_x)\ne 0\}\ .$$
 In particular this support has pure codimension 1. 
\end{th-supp}

The statement   yields strong a priori properties of the first direct image  in a very general framework. This result has interesting consequences:

\newtheorem*{co-tf}{Corollary \ref{tf}} 
\begin{co-tf} If $\BN_\pi({\cal E})$ has codimension $\geq 2$ at any point, then $R^1\pi_*({\cal E})$ is torsion free. In particular the singularity set of this sheaf has codimension $\geq 2$ at any point.
\end{co-tf}

 Equivalently,  
\newtheorem*{co-div}{Corollary \ref{div}}
\begin{co-div} If  $R^0\pi_*({\cal E})=0$ and $\Tors(R^1\pi_*({\cal E}))\ne 0$, then  the Brill-Noether locus $\BN_\pi({\cal E})$ contains a non-empty effective divisor.
\end{co-div}
 
Another result concerns the natural question: supposing that we are in the conditions of Corollary \ref{tf}, how far is the torsion free sheaf $R^1\pi_*({\cal E})$ from being a free ${\cal O}_B$-module? The proposition below shows that  the Brill-Noether locus $\BN_\pi({\cal E})$   can be regarded as an obstruction to trivializing globally the torsion free sheaf $R^1\pi_*({\cal E})$.

\newtheorem*{pr-forGeorges}{Proposition \ref{forGeorges}} 
\begin{pr-forGeorges} Let $\pi:M\to B$ be a proper, surjective holomorphic submersion with connected base $B$, connected surfaces as fibers, and ${\cal E}$ a holomorphic bundle on  $M$ and such that 
\begin{enumerate}
\item The Brill-Noether locus $\BN_\pi({\cal E})$  
has codimension $\geq 2$ at every point. 
\item The map $B\ni x\mapsto h^2({\cal E}_x)\in\N$ is constant. 
\end{enumerate}
Put $k:=\rk(R^1\pi_*({\cal E}))$ and assume that
$$s=(s_1,\dots,s_k)\in H^0(B,R^1\pi_*({\cal E}))^{\oplus k}$$
is a system of sections with the property that $s(x_0)$ is linearly independent in the fiber $(R^1\pi_*({\cal E}))(x_0)$ for a point $x_0\in B\setminus \BN_\pi({\cal E})$. Put
 $${\cal S}:=\{x\in B\setminus\BN_\pi({\cal E})|\ s(x)\hbox{ is linearly dependent in }(R^1\pi_*({\cal E}))(x)\}\ .
 $$
 Then the closure $\bar {\cal S}$ of ${\cal S}$ is an effective divisor containing $\BN_\pi({\cal E})$. In particular this divisor is non-empty if $\BN_\pi({\cal E})$ is non-empty.
\end{pr-forGeorges}

In sections \ref{inductive} and \ref{ext} we give two explicit descriptions   of the torsion  subsheaf   $\Tors(R^1\pi_*({\cal E}))$.  The first result concerns only the case $ R^0\pi_*({\cal E})=0$ and identifies this torsion sheaf with an inductive limit indexed by the ordered set ${\cal D}iv(B)$ of effective divisors of the base $B$.

The second statement is more general, more precise, and implies the first. On the other hand, whereas  the second result uses formal algebraic homological techniques, the first one is proved with classical complex geometric methods. Therefore  we believe that both approaches are  interesting. 

\newtheorem*{th-indlim}{Theorem \ref{indlim}} 
\begin{th-indlim}  Suppose that $\pi$ has connected fibers and $R^0\pi_*({\cal E})=0$. Then there exists a canonical isomorphism

$$\varinjlim_{\substack{D\in{\cal D}iv(B) \\  |D|\subset \BN_\pi({\cal E})}} R^\pi_D=\varinjlim_{D\in{\cal D}iv(B)} R^\pi_D\textmap{\simeq \ \psi} \Tors(R^1\pi_*({\cal E}))\ ,
$$
where  ${\cal D}iv(B)$ denotes the ordered set of effective divisors on $B$ and, for an effective divisor $D\subset B$, we put ${\cal D}:=\pi^{-1}(D)$ and 
$$R_D^\pi:=\pi_*({\cal E}_{\cal D})(D)=\pi_*({\cal E}_{\cal D}({\cal D}))\in{\cal C}oh(B)\ . 
$$
\end{th-indlim}

Let $\pi:M\to B$ be a proper morphism of complex spaces  and ${\cal E}$ a coherent sheaf on $M$ which is flat over $B$. According to Corollary 4.11 p. 133 \cite{BS} there exists a coherent sheaf ${\cal T}_{\cal E}$ on $B$, unique up to isomorphism with the following property:   for every coherent sheaf ${\cal A}$ defined on an open set $U\subset B$,  there exists an isomorphism $\pi_*({\cal E}\otimes \pi^*({\cal A}))={\cal H}om({\cal T}_{\cal E},{\cal A})$ which is functorial with respect to ${\cal A}$. With this notation we can state:

\newtheorem*{th-new-desc}{Theorem \ref{new-desc}} 
\begin{th-new-desc} 
Let $\pi:M\to B$ be a proper morphism of complex spaces  and ${\cal E}$ a coherent sheaf on $M$ which is flat over $B$. Suppose that $B$ is locally irreducible. Then
\begin{enumerate}[(i)]
\item One has canonical isomorphisms
$${\cal E}xt^1({\cal T},{\cal O}_B)\textmap{\simeq}\Tors(R^1\pi_*({\cal E}))\ .
$$
\item If $\pi_*({\cal E})=0$ and $B$ is smooth, then
\begin{enumerate}
\item ${\cal T}$ is a torsion sheaf,
\item Denoting by $D_{\max}$ the maximal divisor contained in the complex subspace defined by   $\Ann({\cal T})$, one has canonical isomorphisms
$${\cal H}om({\cal T}_{D_{\max}},{\cal O}_{D_{\max}})\textmap{\simeq}{\cal H}om({\cal T},{\cal O}_{D_{\max}})\textmap{\simeq}\Ext^1({\cal T},{\cal O}_B)\textmap{\simeq}\Tors(R^1\pi_*({\cal E}))\ .
$$
\end{enumerate} 
 
\end{enumerate}

\end{th-new-desc}

These results have effective applications. For instance Corollary \ref{div} can be viewed as a tool to prove  existence of divisors on a given complex manifold. More precisely, let $B$ be a complex surface (for instance a class VII surface) and   ${\cal E}$ the universal bundle on  $B\times Y$ associated with an embedding  of a compact complex manifold $Y$ in a moduli space of simple bundles on $B$ \cite{Te3}. We conclude that,  if $\Tors(R^1\pi_*({\cal E}))\ne 0$, then $B$ has curves.

Proposition \ref{forGeorges} has been used in \cite{D2} for studying deformations of class VII surfaces, and is applied taking for ${\cal E}$ the relative tangent bundle of the deformation.   The author defines an explicit family of class VII surfaces with $b_2=b$ parameterized by an open set   $B\subset\C^{2b}$ which is generically versal. This family contains surfaces which admit (non-trivial) global tangent vector fields, so in this case the Brill-Noether locus is non-empty.  Therefore Proposition \ref{forGeorges} applies and gives a non-empty divisor in $B$ containing the points which correspond to surfaces with global tangent vector fields and the points where the family is non-versal. 
 An important problem in \cite{D2} is to determine this divisor explicitly.

Finally, Theorem \ref{indlim} can be used to ``compute" $H^0(\Tors(R^1\pi_*({\cal E})))$ supposing that the restrictions of ${\cal E}$ to the ``vertical" divisors ${\cal D}:=\pi^{-1}(D)$ are known. 
Note that, under the assumption $R^0\pi_*({\cal E})=0$, the Leray spectral sequence associated with the pair $({\cal E},\pi)$  induces a canonical isomorphism $H^1({\cal E})\textmap{\simeq} H^0(R^1\pi_*({\cal E}))$.  Therefore   $H^0(R^1\pi_*({\cal E}))$ (in particular its subspace $H^0(\Tors(R^1\pi_*({\cal E})))$) is relevant for the computation of the cohomology space $H^1({\cal E})$.
We will discuss   these applications in section \ref{App}.

 \section{First properties of $R^1\pi_*({\cal E})$}

In this section we prove the first general properties of the sheaf $R^1\pi_*({\cal E})$ under the assumption $R^0\pi_*({\cal E})=0$.  As usually we will use the same notation for a holomorphic bundle and the associated locally free sheaf of local sections. 

Let $B$, $M$ be complex manifolds, $\pi:M\to B$ a proper, surjective holomorphic submersion, and let ${\cal E}$ be a holomorphic bundle on  $M$.  
Let $U\subset B$ be an open set, $\varphi\in {\cal O}(U)$ a non-trivial holomorphic function, and $D:=Z(\varphi)$ the associated effective divisor.   Let $m_\varphi: \resto{R^1\pi_*({\cal E})}{U}\to \resto{R^1\pi_*({\cal E})}{U}$ be the morphism defined by multiplication with $\varphi$. By the definition of the ${\cal O}_B$-module structure on $R^1\pi_*({\cal E})$, the morphism $m_\varphi$ is just $R^1\pi_*(m_\Phi)$, where  $m_\Phi$ is the morphism of sheaves $\resto{{\cal E}}{\pi^{-1}(U)}\to \resto{{\cal E}}{\pi^{-1}(U)}$ defined by multiplication with the function
$$\Phi:=\pi^*(\varphi)=\varphi\circ\pi\in {\cal O}(\pi^{-1}(U))\ .$$ 

Tensorising by the locally free sheaf ${\cal E}$ the tautological exact sequence associated with the pull-back divisor ${\cal D}=Z(\Phi)$, we get the short exact sequence
$$0\map \resto{{\cal E}}{\pi^{-1}(U)}\textmap{\Phi\cdot}  \resto{{\cal E}}{\pi^{-1}(U)}\map {\cal E}_{\cal D}\map 0
$$
which yields
\begin{equation}\label{les} 0\map  \pi_*(\resto{{\cal E}}{\pi^{-1}(U)})\textmap{\varphi\cdot }   \pi_*(\resto{{\cal E}}{\pi^{-1}(U)})\map
\phantom{MMMMMMMMMMMMM}$$
$$\phantom{MMMMMMM}\map    \pi_*({\cal E}_{\cal D})\to R^1\pi_*(\resto{{\cal E}}{\pi^{-1}(U)})\textmap{R^1\pi_*(m_\Phi)}  R^1\pi_*(\resto{{\cal E}}{\pi^{-1}(U)})\dots
\end{equation}

Denote by $j$ and $J$ the inclusions of $D$ and ${\cal D}$ in $B$ and $M$ respectively.
The sheaf ${\cal E}_{\cal D}$ can be written as $J_*(\resto{{\cal E}}{{\cal D}})$ hence,   since $\pi\circ J=j\circ (\resto{\pi}{{\cal D}})$ we obtain
\begin{equation}\label{J}
\pi_*({\cal E}_{\cal D})=\pi_*(J_*(\resto{{\cal E}}{{\cal D}}))=(\pi\circ J)_*(\resto{{\cal E}}{{\cal D}})=(j\circ \resto{\pi}{{\cal D}})_*(\resto{{\cal E}}{{\cal D}})=j_*\left((\resto{\pi}{{\cal D}})_*(\resto{{\cal E}}{{\cal D}})\right)\ .
\end{equation}
\begin{dt} The Brill-Noether locus of the pair $(\pi,{\cal E})$ is defined by
 $$\BN_\pi({\cal E}):=\{x\in B |\ h^0({\cal E}_x)\ne 0\}\subset B\ ,
$$
where, for $x\in B$ we denoted by ${\cal E}_x$ the restriction of ${\cal E}$ to the fiber $M_x:=\pi^{-1}(x)$.
\end{dt}
\begin{lm}\label{red} Suppose that the divisor $Z(\varphi)$ is reduced, and that $\BN_\pi({\cal E})\cap Z(\varphi)$ has codimension $\geq 2$ at every point. Then 
$$\ker(m_\varphi:\resto{R^1\pi_*({\cal E})}{U}\to \resto{R^1\pi_*({\cal E})}{U})=0\ .
$$
\end{lm}
\begin{proof}  Taken into account the exact sequence (\ref{les}) and formula (\ref{J})   it suffices to prove that 
$(\resto{\pi}{{\cal D}})_*(\resto{{\cal E}}{{\cal D}})=0$.
  Let $V\subset D:=Z(\varphi)$ be an open set and $W$ its pre-image in ${\cal D}$. One has
$$(\resto{\pi}{{\cal D}})_*(\resto{{\cal E}}{{\cal D}})(V)=H^0(W, \resto{{\cal E}}{{\cal D}})\ .
$$
Since $W$ is reduced, the vanishing of a section $s\in H^0(W, \resto{{\cal E}}{{\cal D}})$ can be tested pointwise. But the restriction of any such section to the dense set 
$$W\setminus (\resto{\pi}{{\cal D}})^{-1}(\BN_\pi({\cal E}))$$
 vanishes obviously (because it vanishes fiberwise).  This shows $H^0(W, \resto{{\cal E}}{{\cal D}})=0$.
\end{proof}

Using Lemma \ref{red} we can prove now our first result about the torsion of $R^1\pi_*({\cal E})$:
\begin{thry}\label{supp}
Let $B$, $M$ be complex manifolds, $\pi:M\to B$ a proper, surjective holomorphic submersion, and let ${\cal E}$ be a holomorphic bundle on  $M$. If $R^0\pi_*({\cal E})=0$ then the support  $\supp(\Tors(R^1\pi_*({\cal E})))$ coincides with the maximal pure 1-codimensional analytic subset $D_\pi({\cal E})\subset B$   contained  in  $\BN_\pi({\cal E})$.
 In particular this support has pure codimension 1. 
\end{thry}

\begin{proof}
Let $x\in\supp(\Tors(R^1\pi_*({\cal E})))$. Therefore there exists   $u\in R^1\pi_*({\cal E})_x\setminus\{0\}$ and a germ $\varphi_x=(x,U\textmap{\varphi}\C)\in{\cal O}_x\setminus\{0\}$ such that $\varphi_x u=0$. We may suppose that $\varphi_x$ is irreducible in the local ring ${\cal O}_x$. Recall that the set of points at which a complex space is reduced is Zariski open. This is a consequence of Cartan's coherence theorem for the ideal sheaf of nilpotent elements (see \cite{KK} 47.II, E 47 c p. 182). Therefore there exists an open neighborhood of $x$ in  $Z(\varphi)$ which is reduced. In other words, replacing $U$ by a smaller open neighborhood of $x$ in $B$ if necessary,  we may suppose that the effective divisor $D:=Z(\varphi)\subset U$ is reduced\footnote{Note that, in complex analytic geometry, being irreducible at a point is not an open property, as one can see in the example of the Withney umbrella, which is defined by the equation $xy^2 = z^2$. 
}.

Using Lemma \ref{red} and taking into account  that $D$ is irreducible at $x$, we conclude that locally around $x$ the divisor $D$ is  contained in  $\BN_\pi({\cal E})$. In other words, taking $U$ sufficiently small we will have $D\subset \BN_\pi({\cal E})$, hence $x\in D_\pi({\cal E})$. This proves the inclusion $\supp(\Tors(R^1\pi_*({\cal E})))\subset D_\pi({\cal E})$. Conversely let $x\in D_\pi({\cal E})$  and $\varphi_x=(x,U\textmap{\varphi}\C)\in{\cal O}_x\setminus\{0\}$ an irreducible germ defining an irreducible component of $D_\pi({\cal E})$ at $x$. Taking $U$ sufficiently small we may assume that $D:=Z(\varphi)$ is reduced. Using again the exact sequence  (\ref{les}) we see that the sheaf $\pi_*({\cal E}_{\cal D})=j_*\left((\resto{\pi}{{\cal D}})_*(\resto{{\cal E}}{{\cal D}})\right)$ is mapped injectively into $\Tors(R^1\pi_*({\cal E}))$.  But, by Grauert's theorems, the sheaf  $(\resto{\pi}{{\cal D}})_*(\resto{{\cal E}}{{\cal D}})$ has positive rank on $D$. Therefore $D\subset \supp(\Tors(R^1\pi_*({\cal E})))$, which proves the inclusion $D_\pi({\cal E})\subset \supp(\Tors(R^1\pi_*({\cal E})))$.

\end{proof}

\begin{co} \label{tf} Suppose that the Brill-Noether locus $\BN_\pi({\cal E})$
of the pair $(\pi,{\cal E})$ has codimension $\geq 2$ at every point. Then the first direct image $R^1\pi_*({\cal E})$ is torsion free.
\end{co}

We also state explicitly the following obvious reformulation of this Corollary \ref{tf}, which can be regarded as {\it a criterion for existence of divisors} in the base of the fibration:
\begin{co} \label{div} If  $R^0\pi_*({\cal E})=0$ and $\Tors(R^1\pi_*({\cal E}))\ne 0$, then  the Brill-Noether locus $\BN_\pi({\cal E})$ contains a non-empty effective divisor.
\end{co}

A similar statement is obtained if one replaces the condition $R^0\pi_*({\cal E})=0$ with the condition ``cohomologically flat in dimension 0" (see \cite{BS} p. 133-134):

 \begin{pr} \label{tfnew} Suppose that $B\ni x\mapsto h^0({\cal E}_x)\in\N$ is constant. Then $R^1\pi_*({\cal E})$ is torsion free.
\end{pr}
\begin{proof} Since the map $B\ni x\mapsto h^0({\cal E}_x)\in\N$ is constant, $\pi_*({\cal E})$ is locally free and commutes with base change by Grauert's theorems (see \cite{BS} Theorems 4.10(d), 4.12). Here we used the properness and the flatness of $\pi$ (which implies the flatness of ${\cal E}$ over $B$). Using the base change property we see that the natural morphism
$$\resto{\pi_*(\resto{{\cal E}}{\pi^{-1}(U)})}{D}
 \map   (\resto{\pi}{{\cal D}})_*({\cal E}_{\cal D})
$$
is an isomorphism. Applying the functor $(i_D)_*$ to the two sheaves, it follows that the natural morphism
$$ \pi_*(\resto{{\cal E}}{\pi^{-1}(U)})\otimes{\cal O}_D \to \pi_*({\cal E}_{\cal D})$$
is an isomorphism. Via this isomorphism the morphism %
$ \pi_*(\resto{{\cal E}}{\pi^{-1}(U)})\to  \pi_*({\cal E}_{\cal D})$
 in the exact sequence  (\ref{les}) corresponds to the canonical epimorphism  
$$ \pi_*(\resto{{\cal E}}{\pi^{-1}(U)})\to \pi_*(\resto{{\cal E}}{\pi^{-1}(U)})\otimes{\cal O}_D$$
 so is surjective.
Therefore 
$$ \ker(m_\varphi:\resto{R^1\pi_*({\cal E})}{U}\to \resto{R^1\pi_*({\cal E})}{U})=\{0\}$$
 by the exact sequence (\ref{les}). 

An alternative proof can be obtained using the definition of cohomologically flatness in dimension 0. For any  open set $U\subset B$ and any non-trivial holomorphic map $\varphi\in {\cal O}^*(U)$ we get a monomorphism $0\to {\cal O}(U)\textmap{\varphi}{\cal O}(U)$ so, by   Theorem 4.10 $(a')$ \cite{BS} the induced morphism $m_\varphi:\resto{R^1\pi_*({\cal E})}{U}\to \resto{R^1\pi_*({\cal E})}{U}$ is injective.
\end{proof}

Corollary  \ref{tf} and Proposition \ref{tfnew} give criteria which guarantee   the first direct image $R^1\pi_*({\cal E})$ being torsion free. The proposition below gives a criterium which guarantees   this sheaf being  a  free ${\cal O}_X$-module.
\begin{pr}\label{free} Let $\pi:M\to B$ be a proper, surjective holomorphic submersion with connected surfaces as fibers, and ${\cal E}$ a holomorphic bundle on  $M$   such that 
\begin{enumerate}
\item The Brill-Noether locus $\BN_\pi({\cal E}):=\{x\in B |\ h^0({\cal E}_x)\ne 0\}$
has codimension $\geq 2$ at every point. 
\item The map $B\ni x\mapsto h^2({\cal E}_x)\in\N$ is constant. 
\end{enumerate}
Put $k:=\rk(R^1\pi_*({\cal E}))$ and assume that $s=(s_1,\dots,s_k)\in H^0(B,R^1\pi_*({\cal E}))^{\oplus k}$ is a system of sections such that $s(x)$ is linearly independent in the fiber $(R^1\pi_*({\cal E}))(x)$ for every $x\in B\setminus \BN_\pi({\cal E})$.    Then  
\begin{enumerate}
\item  The morphism $\sigma:{\cal O}_B^{\oplus k}\to R^1\pi_*({\cal E})$ defined by $s$ is an isomorphism, in particular
  $R^1\pi_*({\cal E})$  is a free ${\cal O}_X$-module and  $s(x)$  is linearly independent in   $(R^1\pi_*({\cal E}))(x)$ for every $x\in B$.
  \item $\BN_\pi({\cal E})=\emptyset$.
\end{enumerate}

\end{pr}
Therefore, under the assumptions of the theorem,    the first direct image $R^1\pi_*({\cal E})$ is free if this sheaf admits a system $s$ of $k:=\rk(R^1\pi_*{(\cal E}))$ sections  which are linearly independent on the open set $B\setminus \BN_\pi({\cal E})$ (on which $R^1\pi_*({\cal E})$ is locally free).\\

\begin{proof}    Since the map $B\ni x\mapsto h^2({\cal E}_x)\in\N$ is constant, it follows by Grauert's theorems that $R^2\pi_*({\cal E})$ is locally free and that $R^2\pi_*({\cal E})$, $R^1\pi_*({\cal E})$ commute with base changes (\cite{BS} Theorem 3.4 p. 116). Therefore the canonical morphisms  $(R^i\pi_*({\cal E}))(x)\to H^i({\cal E}_x)$ are isomorphisms for $i=1$, 2, and for every $x\in B$. 

By Riemann-Roch theorem and the second assumption it follows that the map $B\ni x\mapsto h^1({\cal E}_x)$ is constant on $B\setminus  \BN_\pi({\cal E})$, hence the sheaf ${\cal F}:=R^1\pi_*({\cal E})$ (which is torsion free by Proposition \ref{tf}) is locally free on this open subset.
We know by hypothesis that the  morphism $\sigma:{\cal O}_B^{\oplus k}\to {\cal F}$ defined by $s$ is  a bundle isomorphism on $U:=B\setminus  \BN_\pi({\cal E})$.   The canonical embedding $c:{\cal F}\to({\cal F}^\vee)^\vee$  of ${\cal F}$ in its bidual sheaf is also an isomorphism on $U$, so the composition $c\circ\sigma$ has this property too. The inverse $\tau:=\left\{\resto{c\circ\sigma}{U}\right\}^{-1}$ can be regarded as a section of $\Hom(({\cal F}^\vee)^\vee, {\cal O}_B^{\oplus k})=\left\{{\cal F}^\vee\right\}^{\oplus k}$ defined on $U$.  Since $\left\{{\cal F}^\vee\right\}^{\oplus k}$ is a reflexive sheaf  and $U$ is the complement of a Zariski closed subset of codimension $\geq 2$,  it follows that $\tau$ extends to a global morphism  $\tilde\tau: ({\cal F}^\vee)^\vee\to {\cal O}_B^{\oplus k}$. We have 
$$\tilde\tau\circ (c\circ\sigma)=\id_{{\cal O}_B^{\oplus k}}\ ,\ (c\circ\sigma)\circ \tilde\tau=\id_{({\cal F}^\vee)^\vee}$$
 because these equalities hold on $U$. This shows that $c\circ\sigma$ is an isomorphism, in particular $\sigma$ is a monomorphism and  $c$ is an epimorphism. Since $c$ is also a monomorphism,  it follows that $c$ is an isomorphism, so $\sigma=c^{-1}\circ(c\circ\sigma)$ will also be an isomorphism. 
 
 For the second statement, recall that  the canonical map $(R^1\pi_*({\cal E}))(x)\to H^1({\cal E}_x)$ is  an isomorphism for every $x\in B$, hence the map $x\mapsto h^1({\cal E}_x)$ is constant on $B$ (and coincides with $k$). Using the second assumption of the hypothesis and the Riemann-Roch theorem, we conclude that the map $x\mapsto h^0({\cal E}_x)$ is constant on $B$, hence it vanishes identically.
\end{proof}

Proposition \ref{free} leads naturally to the question: if one assumes that the system $s$ is {\it  generically} linearly independent and  the Brill-Noether locus $\BN_\pi({\cal E})$ is non-empty, what can be said about the set of points where $s$ is linearly dependent? The answer is:
\begin{pr}\label{forGeorges}
Let $\pi:M\to B$ be a proper, surjective holomorphic submersion with connected base $B$, connected surfaces as fibers, and ${\cal E}$ a holomorphic bundle on  $M$  such that 
\begin{enumerate}
\item The Brill-Noether locus $\BN_\pi({\cal E})$  
has codimension $\geq 2$ at every point. 
\item The map $B\ni x\mapsto h^2({\cal E}_x)\in\N$ is constant. 
\end{enumerate}
Put $k:=\rk(R^1\pi_*({\cal E}))$ and assume that
$$s=(s_1,\dots,s_k)\in H^0(B,R^1\pi_*({\cal E}))^{\oplus k}$$
is a system of sections such that $s(x_0)$ is linearly independent in the fiber $(R^1\pi_*({\cal E}))(x_0)$ for a point $x_0\in B\setminus \BN_\pi({\cal E})$. Put
 $${\cal S}:=\{x\in B\setminus\BN_\pi({\cal E})|\ s(x)\hbox{ is linearly dependent in }(R^1\pi_*({\cal E}))(x)\}\ .
 $$
 Then the closure $\bar {\cal S}$ of ${\cal S}$ is an effective divisor containing $\BN_\pi({\cal E})$. In particular this divisor is non-empty if $\BN_\pi({\cal E})$ is non-empty.
\end{pr}
\begin{proof}
Using the notations introduced in the proof of Proposition \ref{free}, regard the wedge product $\wedge\sigma:=s_1\wedge\dots\wedge s_k$ as a section in $\wedge^k({\cal F})$. Let %
$$d:\wedge^k({\cal F})\to \{\wedge^k({\cal F})^\vee\}^\vee=\det({\cal F})$$
 be the canonical embedding (see Proposition 6.10 ch. V in \cite{K}) and $D$ the vanishing divisor of the section $d\circ (\wedge\sigma)$ of this holomorphic line bundle. Note that $D\cap (B\setminus\BN_\pi({\cal E}))={\cal S}$ which (taken into account that $\BN_\pi({\cal E})$ has codimension $\geq 2$ at every point) implies $\bar {\cal S}=D$. In order to prove that $\BN_\pi({\cal E})\subset D$ we will show that $B\setminus D\subset B\setminus \BN_\pi({\cal E})$. Let  $x_0\in B\setminus D$ and let $U$ be an open neighborhood of $x_0$ such that $U\cap D=\emptyset$. This implies $U\cap  {\cal S}=\emptyset$, hence $\resto{s}{U}$ satisfies the hypothesis of Proposition \ref{free}, which gives $\BN_\pi({\cal E})\cap U=\emptyset$, hence $x\not\in \BN_\pi({\cal E})$.

\end{proof}

\begin{re}
Note that, in general, for a point $x_0\in \BN_\pi({\cal E})$, the system $s(x_0)$ can be linearly independent in $(R^1\pi_*({\cal E}))(x_0)$ although, by Proposition \ref{forGeorges},   such a point belongs to the closure  of the set ${\cal S}$ of points in  $B\setminus\BN_\pi({\cal E})$ where $s(x)$ is linearly dependent. This shows that, in general, for a system of sections in a torsion free coherent sheaf, being fiberwise linearly dependent is not always a closed condition.
\end{re}
{\ }\\
{\bf Example:} Let ${\cal F}$ be the ideal sheaf ${\cal I}_0$ of the origin in $\C^2$, and let $s\in H^0({\cal I}_0)$ be the section defined by the holomorphic function $(z_1,z_2)\mapsto z_1$. Then $s(0)$ is non-zero in the fiber ${\cal I}_0(0)$, although $s(z)$ vanishes in ${\cal I}_0(z)$ for every point $z\in \{0\}\times\C^*$. \\

In section \ref{App} we will see that Proposition \ref{forGeorges} has found interesting applications in studying families of class VII surfaces \cite{D2}.

\section{The sheaf  $\Tors(R^1\pi_*({\cal E}))$ as inductive limit.}\label{inductive}

Let again $B$, $M$ be complex manifolds, $\pi:M\to B$ a proper, surjective holomorphic submersion, and let ${\cal E}$ be a holomorphic bundle on  $M$. For an effective divisor $D\subset B$ we put ${\cal D}:=\pi^{-1}(D)$, which is an effective divisor of $M$.   The restriction ${\cal E}_{\cal D}$ will be alternatively regarded either as a   locally free sheaf on ${\cal D}$, or as a torsion coherent sheaf on $M$.

Recall that for a decomposition $D=D'+D''$ of an effective divisor $D$ of $B$ as sum of effective divisors we have an associated {\it decomposition exact sequence}
\begin{equation}\label{dec-ex-seq}
0\to{\cal O}_{D'}(-D'')\to {\cal O}_D\to{\cal O}_{D''}\to 0
\end{equation}
induced by the restriction morphism $r_{DD''}:{\cal O}_D\to{\cal O}_{D''}$ and the obvious isomorphisms 
$$\ker(r_{DD''})=\qmod{{\cal I}_{D''}}{{\cal I}_D}=\qmod{{\cal I}_{D''}}{{\cal I}_{D'}{\cal I}_{D''}}=\qmod{{\cal O}(-D'')}{ {\cal I}_{D'}{\cal O}(-D'')}={\cal O}_{D'}(-D'')\ .
$$

\begin{lm}\label{divisor} Let $D=\sum_{i\in I} n_i D_i$ be an  effective  divisor of $B$ decomposed as sum of irreducible components, such that  $D_i\not\subset \BN_\pi({\cal E})$ for every $i\in I$. Then $H^0({\cal E}_{{\cal D}})= 0$.
\end{lm}
\begin{proof} The statement is clear when $D$ is irreducible, because a section    in a holomorphic bundle over an irreducible space vanishes if it vanishes generically (as a map with values in the total space of the bundle). On the other hand  if $D\not\subset \BN_\pi({\cal E})$, any section of ${\cal E}_{\cal D}$ (regarded as holomorphic bundle over ${\cal D}$) will vanish   on the non-empty Zariski open subset ${\cal D}\setminus\pi^{-1}(\BN_\pi({\cal E}))$ of ${\cal D}$. 

For $D$ reducible we use induction with respect to $n:=\sum_{i\in I} n_i$. We choose $i_0\in I$, we put $D'':=D-D_{i_0}$,   and we use the short exact sequence of sheaves on $M$:
$$0\to {\cal E}_{ {\cal D}_{i_0}}(-{\cal D}'')\to  {\cal E}_{{\cal D}}\to {\cal E}_{ {\cal D}''}\to 0\ .
$$
Noting that ${\cal E}(-{\cal D}'')= {\cal E}\otimes\pi^*({\cal O}_B(-D''))$, we see that the two bundles have the same Brill-Noether locus over $B$, hence $H^0({\cal E}_{\cal D_{i_0}}(-{\cal D}''))=0$, by the first step applied to the bundle ${\cal E}(-{\cal D}'')$ on $M$ and the irreducible divisor $D_{i_0}$. On the other hand  $H^0({\cal E}_{ {\cal D}''})=0$ by the induction assumption.
\end{proof}
For an effective divisor $D\subset B$ we put
$$R_D^\pi:=\pi_*({\cal E}_{\cal D})(D)=\pi_*({\cal E}_{\cal D}({\cal D}))\in{\cal C}oh(B)\ . 
$$

\begin{dt} We denote by ${\cal D}iv(B)$ the small category associated with the ordered set of effective divisors of $B$. The morphisms in this category correspond to inclusions of effective divisors (regarded as complex subspaces of $B$).
\end{dt}
\begin{pr}\label{inBN} The assignment $D\mapsto R_D^\pi$ defines a functor 
$$R^\pi:{\cal D}iv(B)\to {\cal C}oh(B)$$
 with the following properties.
\begin{enumerate}
\item The support of $R_D^\pi$ is contained in $D$, in particular $R_D^\pi$ is a torsion sheaf.
\item For $D'\leq D$ the corresponding morphism $R_{D'D}^\pi:R_{D'}^\pi\to R_{D}^\pi$ is a sheaf monomorphism.
\item  $R_D$ depends only on the part of $D$ which is contained in $\BN_\pi({\cal E})$, more precisely consider the decomposition  $D=\sum_{i\in I} n_i D_i$ of $D$ in irreducible  components, let $I_0\subset I$ be the subset of indices $i$ for which $D_i\subset \BN_\pi({\cal E})$, and  put $D^0:=\sum_{i\in I_0} n_i D_i$. Then $R_{D^0D}^\pi$ is an isomorphism.
\end{enumerate}
\end{pr}
\begin{proof} The first statement is obvious. For the second, putting 
$D'':=D-D'$ and tensorizing with ${\cal E}({\cal D})$   the  exact sequence  %
$$0\to{\cal O}_{{\cal D}'}(-{\cal D}'')\to {\cal O}_{{\cal D}}\to{\cal O}_{{\cal D}''}\to 0\ ,
$$
associated with the decomposition $D=D'+D''$, we obtain the short exact sequence
$$0\to {\cal E}_{\cal D'}({\cal D'})\to {\cal E}_{\cal D}({\cal D})\to {\cal E}_{\cal D''}({\cal D})\to 0\ .
$$
The injectivity of $R_{D'D}^\pi$  follows applying the left exact functor $\pi_*$ to this short exact sequence.

To prove the third statement we put $D'':=D-D^0$, and we use the short exact sequence
$$0\to {\cal E}_{\cal D^0}({\cal D^0})\to {\cal E}_{\cal D}({\cal D})\to {\cal E}_{\cal D''}({\cal D})\to 0 
$$
obtained similarly. On the other hand, using Lemma \ref{divisor} we obtain 
$$H^0(\pi^{-1}(U),{\cal E}_{\cal D''}({\cal D}))=0$$
for every open set $U\subset B$, because no irreducible component of ${\cal D''}\cap U$ is contained in  $\BN_\pi({\cal E}({\cal D}))$.
\end{proof}
\begin{co} The natural morphism
$$\varinjlim_{\substack{D\in{\cal D}iv(B) \\  |D|\subset \BN_\pi({\cal E})}} R^\pi_D\to \varinjlim_{D\in{\cal D}iv(B)} R^\pi_D$$
is an isomorphism.
\end{co}
\begin{thry}\label{indlim}  Let $\pi:M\to B$ be a proper holomorphic submersion with connected fibers, and ${\cal E}$  a  locally free coherent sheaf  on $M$ such that $\pi_*({\cal E})=0$. There exists a canonical isomorphism
$$\varinjlim_{\substack{D\in{\cal D}iv(B) \\  |D|\subset \BN_\pi({\cal E})}} R^\pi_D=\varinjlim_{D\in{\cal D}iv(B)} R^\pi_D\textmap{\simeq \ \psi} \Tors(R^1\pi_*({\cal E}))\ .
$$
In particular, the inductive limit on the left is a coherent sheaf supported on the divisorial part $D_\pi({\cal E})$ of $\BN_\pi({\cal E})$.
\end{thry}
\begin{proof} For every $D\in  {\cal D}iv(B)$ we define a morphism $\psi_D:R^\pi_D\to \Tors(R^1\pi_*({\cal E}))$ in the following way: the long exact   sequence associated with the short exact sequence
$$0\to {\cal E}\to {\cal E}({\cal D})\to {\cal E}_{\cal D}({\cal D})\to 0
$$
and the left exact functor $\pi_*$ begins with
$$ 0\to \pi_*({\cal E}_{\cal D}({\cal D}))=R_D^\pi\textmap{\partial_D^\pi} R^1\pi_*({\cal E})\textmap{\alpha_D^\pi}R^1\pi_*({\cal E}({\cal D}))\to \dots
$$
Using the identification $R^1\pi_*({\cal E}({\cal D}))=R^1\pi_*({\cal E})(D)=R^1\pi_*({\cal E})\otimes_{{\cal O}_B}{\cal O}_B(D)$ we see that $\alpha_D^\pi$ is given by multiplication with the canonical morphism $\sigma_D:{\cal O}_B\to {\cal O}_B(D)$. A local equation $\varphi$ of $D$ defined on an open set $U\subset B$ defines an isomorphism ${\cal O}_U(D\cap U)\simeq {\cal O}_U$, and via this isomorphism the restriction 
$$\sigma_D^U:{\cal O}_U\to {\cal O}_U(D\cap U)={\cal O}_U$$
 of $\sigma_D$ to $U$ is given by multiplication with $\varphi$. This shows that $\partial_D^\pi$ defines a monomorphism $\psi_D:R_D^\pi\to \Tors(R^1\pi_*({\cal E}))$ whose image is the annihilator (in $R^1\pi_*({\cal E})$) of the ideal sheaf of  $D$.

For $D'$, $D \in {\cal D}iv(B)$ such that $D'\leq D$ we use the functoriality of the connecting operator $\partial$ with respect to morphisms of short exact sequences,  and we get %
$$\partial_{D}^\pi\circ R_{D'D}^\pi=\partial_{D'}^\pi\ ,$$ 
i.e., $\psi_D\circ R_{D'D}^\pi=\psi_{D'}$. Using the universal property of the inductive limit we see that the system $(\psi_D)_{D\in{\cal D}iv(B)}$ induces a morphism  
$$\psi:\varinjlim_{D\in{\cal D}iv(B)} R^\pi_D\to \Tors(R^1\pi_*({\cal E}))\ ,$$
which is injective, because all morphisms $\psi_D$ are injective. 

The surjectivity of $\psi$ is more delicate: let $x\in B$ and $\chi$ an element in the stalk
$R^1\pi_*({\cal E})_x$ which is a torsion element of this ${\cal O}_{B,x}$-module. Therefore there exists an open neighborhood $U$ of $x$ and a representative $h\in R^1\pi_*({\cal E})(U)$ of $\chi$ and $\varphi\in{\cal O}(U)\setminus\{0\}$ such that $\varphi h=0$ in $R^1\pi_*({\cal E})(U)$. Let $D_U:=Z(\varphi)\in{\cal D}iv(U)$ be the effective divisor defined by $\varphi$ and ${\cal D}_U=Z(\varphi\circ\pi_U)$ its pull-back to $M_U:=\pi^{-1}(U)$ via $\pi_U:=\resto{\pi}{U}$. We have a short exact sequence
$$0\to {\cal E}_{M_U}\textmap{\varphi\circ\pi_U} {\cal E}_{M_U}\map {\cal E}_{{\cal D}_U}={\cal E}_{{\cal D}_U}({\cal D}_U)\to 0\ .
$$
Here we used the trivialization of ${\cal O}_{M_U}({\cal D}_U)$ defined by $\varphi\circ\pi_U$.  The long exact sequence associated with this exact sequence and   the left exact functor $(\pi_U)_*$  contains the segment
$$0\to (\pi_U)_*({\cal E}_{{\cal D}_U}({\cal D}_U))=R_{D_U}^U\stackrel{\partial_{D_U}^U}{\map} R^1\pi_*({\cal E}_{M_U})\stackrel{\alpha_{D_U}^U}{\map}R^1\pi_*({\cal E}_{M_U}({\cal D}_U))=R^1\pi_*({\cal E}_{M_U}) \ ,
$$
in which $\alpha_{D_U}^U$ is given by multiplication  with $\varphi$ in the ${\cal O}_U$-module $R^1\pi_*({\cal E}_{M_U})$. Therefore, replacing $U$ by a smaller open neighborhood of $x$ if necessary, we conclude that $h$ belongs to the image of
$\partial_{D_U}^U$. The problem is that $D_U$ will not necessarily extend to a divisor of $B$, so this does not prove the surjectivity of $\psi$ yet.

On the other hand we know by Proposition \ref{inBN} that $R_{D_U}^U=R_{D_U^0}^U$, where the support of $D_U^0$ is contained in the maximal 1-codimensional analytic subset  $D_{\pi_U}({\cal E}_{M_U})$ of $\BN_{\pi_U}({\cal E}_{M_U})$. The point is that $\BN_{\pi_U}({\cal E}_{M_U})=\BN_{\pi}({\cal E}_{M})\cap U$, so  the similar relation   $D_{\pi_U}({\cal E}_{M_U})=D_{\pi}({\cal E}_{M})\cap U$ holds for the   maximal 1-codimensional analytic subsets of these analytic sets. Therefore there exists a divisor $D$ whose support is contained in $D_{\pi}({\cal E}_{M})$ such that $D\cap U\geq D^0_U$. Note that in general one cannot obtain equality, because different irreducible components of
$D_{\pi_U}({\cal E}_{M_U})$ (which can appear with different multiplicities in the decomposition of $D_U^0$) might belong to the same irreducible component of $D_{\pi}({\cal E}_{M})$.

This shows that $h$ belongs to the image of $\psi_D$, which proves the surjectivity of the morphism $\psi$.
\end{proof}

\begin{co}\label{main} In the conditions of Theorem \ref{indlim}, suppose that $B$ is compact. Then one has a canonical isomorphism
$$\varinjlim_{\substack{D\in{\cal D}iv(B) \\  |D|\subset \BN_\pi({\cal E})}} H^0({\cal E}_{\cal D}({\cal D}))\textmap{\Psi} H^0(\Tors(R^1\pi_*({\cal E})))\ .
$$
\end{co}
\begin{proof} It suffices to recall that the functor $\Gamma$   commutes with inductive limits of sheaves on compact spaces (Theorem 3.10.1, p. 162 [Go]) and that 
$$H^0(B,\pi_*({\cal E}_{\cal D}({\cal D})))=H^0(M, {\cal E}_{\cal D}({\cal D}))\ .$$
\end{proof}

\section{$\Tors(R^1\pi_*({\cal E}))$ as  ${\cal E}xt^1 $ and ${\cal H}om$}\label{ext}

Let $A$ be an integral domain .   For an $A$-module $N$ we denote by $\Ann(N)\subset A$ the annihilator ideal of $N$. For an element $\alpha\in A$ we denote by $\Tors_\alpha N\subset N$ the annihilator submodule of $\alpha$ in $N$. In other words 
$$\Tors_\alpha N:=\ker(\alpha\id_N)\ .$$
 We have obviously   
\begin{equation}\label{TorsIndLim}
\Tors N=\varinjlim_{\alpha\in A^* }  \Tors_\alpha N\ ,
\end{equation}
where the set  $A^*:=A\setminus\{0\}$ is ordered with respect to  divisiblity, and the inductive limit is constructed using the family of inclusions $\iota_{\alpha\beta}: \Tors_\alpha N\hookrightarrow \Tors_\beta N$ associated with pairs $\alpha|\beta$. For an element $\alpha\in A^*$ the natural morphism
$$ \qmod{\Hom(N,A)}{\alpha\Hom(N,A)}\to \Hom\left(N,\qmod{A}{\alpha A}\right)\ .\ 
$$
 is injective, hence the left hand module  can be identified with a submodule of the right hand module. The multiplication with  an element $\gamma\in A^*$ defines an injective  morphism 
$$m_\gamma:\qmod{A}{\alpha A}\to \qmod{A}{ \alpha \gamma A}$$
so an injective morphism 
$$\mu_\gamma:\Hom\left(N,\qmod{A}{\alpha A}\right)\to \Hom\left(N,\qmod{A}{\alpha \gamma A}\right)$$
 defined by composition with $m_\gamma$. One has 
 $$\mu_\gamma^{-1} \left(\qmod{\Hom(N,A)}{\alpha \gamma\Hom(N,A)}\right)= \qmod{\Hom(N,A)}{\alpha\Hom(N,A)}$$
  so $\mu_\gamma$ induces an injective  morphism
$$\eta_\gamma: \Qmod{\Hom\left(N,\qmod{A}{\alpha A}\right)}{\left\{\qmod{\Hom(N,A)}{\alpha\Hom(N,A)}\right\}}\ \  \longrightarrow \ \ \ \ \ \ \ \ \ \ \ \ \ \ \ \ \ \ \ \ \ \ \ \ \ \ \ \ \ \ \ \ \ \ \ \ $$
\vspace{3mm}
$$\ \ \ \ \ \ \ \ \ \ \ \ \ \ \ \ \ \ \ \ \ \ \  \ \ \ \ \ \ \ \ \ \ \ \ \ \ \ \ \ \ \ \ \ \ \ \ \ \ \ \ \ \ \ \ \ \ \longrightarrow\ \  \Qmod{\Hom\left(N,\qmod{A}{\alpha \gamma A}\right)}{\left\{\qmod{\Hom(N,A)}{\alpha \gamma\Hom(N,A)}\right\}}\ .
$$

If $\alpha|\beta$ we put  $m_{\alpha\beta}:=m_{\beta/\alpha}$, $\mu_{\alpha\beta}:=\mu_{\beta/\alpha}$, $\eta_{\alpha\beta}:=\eta_{\beta/\alpha}$.

\begin{pr} \label{Ext1NA} Let $N$ be an  $A$-module.  
\begin{enumerate}[(i)]

\item For any $\alpha\in A^*$ one has a canonical isomorphism
$$\Tors_\alpha(\Ext^1(N,A))=\Qmod{\Hom\left(N,\qmod{A}{\alpha A}\right)}{\left\{\qmod{\Hom(N,A)}{\alpha\Hom(N,A)}\right\}}\ .
$$
\item One has a canonical isomorphism 
$$\varinjlim_{\alpha\in A^*} \left\{\Qmod{\Hom\left(N,\qmod{A}{\alpha A}\right)}{\left\{\qmod{\Hom(N,A)}{\alpha\Hom(N,A)}\right\}}\right\} \textmap{\simeq} \Tors\Ext^1(N,A) \ ,
$$
where  the inductive limit is constructed using the family of morphisms $(\eta_{\alpha\beta})_{\alpha |\beta}$.
 \item If $N$ and $\Tors(N)$ are finitely generated, then $\Ann(\Ext^1(N,A))\ne 0$, in particular $\Ext^1(N,A)$ is a   torsion module,  the inductive limit in (ii) computes the whole $\Ext^1(N,A)$ and it stabilizes, i.e.,   there exists $\alpha_0\in A^*$  such that the canonical morphisms
 $$\Qmod{\Hom\left(N,\qmod{A}{\alpha_0 A}\right)}{\left\{\qmod{\Hom(N,A)}{\alpha_0\Hom(N,A)}\right\}}\ \longrightarrow\  \ \ \ \ \ \ \ \ \ \ \ \ \ \ \ \ \ \ \ \ \ \ \   $$
 $$\ \ \ \ \ \ \ \ \ \ \ \   \longrightarrow\  \varinjlim_{\alpha\in A^*} \left\{\Qmod{\Hom\left(N,\qmod{A}{\alpha A}\right)}{\left\{\qmod{\Hom(N,A)}{\alpha\Hom(N,A)}\right\}}\right\} \to  \Ext^1(N,A)
$$
are isomorphisms. 
 
\end{enumerate}

\end{pr}

\begin{proof}
(i) The long exact sequence associated with the short exact sequence
$$0\to A\textmap{ \alpha\cdot } A\to \qmod{A}{\alpha A}\to 0
$$
and the functor $\Hom(N,\cdot)$  reads
$$0\to \Hom(N,A)\textmap{ \alpha\cdot } \Hom(N,A)\to \Hom(N,\qmod{A}{\alpha A})\to\Ext^1(N,A)\textmap{ \alpha\cdot }\Ext^1(N,A)\ ,$$
which gives a canonical isomorphism
$$\Qmod{\Hom\left(N,\qmod{A}{\alpha A}\right)}{\left\{\qmod{\Hom(N,A)}{\alpha\Hom(N,A)}\right\}}\textmap{\simeq} \Tors_\alpha \Ext^1(N,A)\ .
$$
It remains to prove that via these identifications the morphism  $\eta_{\alpha\beta}$ corresponds to the inclusion morphism 
$$\iota_{\alpha\beta}:\Tors_\alpha \Ext^1(N,A)\to \Tors_\beta \Ext^1(N,A)$$
associated with a pair $\alpha|\beta$. Using the morphism of short exact sequences
$$
\begin{diagram}[s=7mm,w=8mm,midshaft]
0&\rTo & A & \rTo^{\alpha\cdot} & A & \rTo & \qmod{A}{\alpha A}&\rTo &0\\
&&\dTo<{\id_A} & & \dTo<{\gamma\cdot}& &\dTo<{\gamma\cdot}\\
0&\rTo & A & \rTo^{\beta\cdot} & A & \rTo & \qmod{A}{\beta A}&\rTo &0
\end{diagram}\ ,
$$
 we get the commutative diagram
$$
\begin{diagram}[s=7mm,w=10mm,midshaft]
  \Hom(N,\qmod{A}{\alpha A}) & \rTo & \Ext^1(N,A) & \rTo^{\alpha\cdot} & \Ext^1(N,A) \\
 \dTo<{\eta_{\alpha\beta}} & & \dTo<{\id}& &\dTo<{\gamma\cdot}\\
 \Hom(N,\qmod{A}{\beta A}) & \rTo & \Ext^1(N,A) & \rTo^{\beta\cdot} & \Ext^1(N,A) \end{diagram}\ ,
$$
which proves the claim.
{\ }\\ \\
(ii) Follows directly from (i) taking into account   (\ref{TorsIndLim}).
\\ \\
(iii) Let $N_0$ be the quotient $N_0:=N/\Tors(N)$, which is torsion free. Since $N_0$ is a finitely generated torsion free module over the integral ring $A$, it can be embedded in a finitely generated free module $F$ such that $\Ann(F/N_0)\ne 0$. We recall briefly the argument. Let $K$ be the field of fractions of $A$ and $\iota: N_0\to V_0:=K\otimes_A N_0$ the canonical embedding.  We identify $N_0$ with its image  $\iota(N_0)$, and we fix a finite $A$-generating set $G$ of $N_0$ (which will also be a $K$-generating set of $V_0$) and a system $(g_1,\dots,g_n)$ of $G$ which is a $K$-basis of $V_0$. Developing  all generators $g\in G$   with respect to this basis, we see that there exists  $\alpha\in A^*$ such that $ N_0\subset \frac{1}{\alpha}(A g_1\oplus\dots \oplus A g_n)$. The $A$-module  $F:=\frac{1}{\alpha}(A g_1\oplus\dots \oplus A g_n)$ is obviously free  and  $\alpha F\subset  N_0$. Therefore   $\alpha(F/N_0)=0$, hence $\alpha\in \Ann(F/N_0)$.
 
The short exact sequence 
$$
0\to N_0\to F\to  F/N_0\to 0
$$ 
yields  an  isomorphism $\Ext^1(N_0,A)\simeq  \Ext^2((F/N_0),A)$. Using the Corollary to Proposition 6, p. A X 89, N° 4 \cite{Bou}, we obtain 
$$\{\alpha\}\subset \Ann(F/N_0)\subset\Ann(\Ext^1(N_0,A))\ .$$
   Therefore $\Ann(\Ext^1(N_0,A))\ne 0$. On the other hand the exact sequence
$$0\to\Tors(N)\to N\to N_0\to 0\ ,\ $$
yields the exact sequence
$$ 0\to \Ext^1(N_0,A)\to \Ext^1(N,A)\to\Ext^1(\Tors(N),A)\ ,
$$
which shows that $\Ann(\Ext^1(N_0,A))\Ann(\Ext^1(\Tors(N),A))\subset\Ann(\Ext^1(N,A))$. We know that $\Ann(\Ext^1(N_0,A))\ne 0$ and, since   $\Tors(N)$ is also assumed finitely generated, we have $\Ann(\Tors(N))\ne 0$, which implies $\Ann(\Ext^1(\Tors(N),A))\ne 0$ (again by Corollary to Proposition 6, p. A X 89, N° 4 \cite{Bou}).  

Therefore,  the annihilator $\Ann(\Ext^1(N,A))$ contains a non-zero element $\alpha_0$, hence
$$\Ext^1(N,A)=\Tors(\Ext^1(N,A))=\Tors_{\alpha_0}(\Ext^1(N,A))\ ,
$$
which can be identified with $\Hom\left(N,\qmod{A}{\alpha_0 A}\right)/\left\{\qmod{\Hom(N,A)}{\alpha_0\Hom(N,A)}\right\}$ by (i).
\end{proof}

 We will see that in the case when  $A$ is a unique factorization domain,  $N$ is a finitely generated  torsion module, and $\Ann(N)$ is a finitely generated ideal, there is an explicit, canonical choice of  an element $\alpha_0\in A^*$ which computes $\Ext^1(N,A)$ as in Proposition \ref{Ext1NA} (iii). \\

Suppose that $A$ is a  unique factorization domain and that $N$ is a finitely generated torsion module. This implies that the annihilator ideal $\Ann(N)$ does not vanish. 
Suppose also that $\Ann(N)$ is finitely generated. This condition holds automatically if we assume $A$ to be Noetherian. We can define $a:=\gcd(\Ann(N))\in A^*$. In other words one has $a=\gcd(\varphi_1,\dots,\varphi_k)$, where $(\varphi_i)$ is a system of  generators of $\Ann(N)$. For any $\beta\in A^*$ we put $(a,\beta):=\gcd(a,\beta)$.

\begin{lm} \label{hope} Suppose that $A$ is a  unique factorization domain, $N$ is a finitely generated torsion module, and $\Ann(N)$ is finitely generated. For any $\beta\in A^*$ one has
$$\Hom(N,\qmod{A}{\beta A})=\mu_{(a,\beta)\beta}\left(\Hom(N,\qmod{A}{(a,\beta) A})\right) 
$$
\end{lm}

\begin{proof}
The inclusion 
$$\mu_{(a,\beta)\beta}\left(\Hom(N,\qmod{A}{(a,\beta) A})\right)\subset \Hom(N,\qmod{A}{\beta A})
$$
is obvious and has been  used before. For the other inclusion let $f\in \Hom(N,\qmod{A}{\beta A})$.  Denote by ${\cal J}$ the ideal of $A$ defined as the pre-image in $A$ of $f(N)$.
One has
\begin{equation}\label{incl}\Ann(N){\cal J}\subset \beta A\subset {\cal J}\ ,
\end{equation}
where the first inclusion follows from $\Ann(N)f(N)=0$. We claim that
\begin{equation}\label{Jideal}
{\cal J}\subset (\beta/(a,\beta)) A\ .
\end{equation}

Indeed, for a prime element $p\in A$ let $p^m$ be its maximal power which divides all  generators $\varphi_i$ of $\Ann(N)$. In other words, $p^m$ divides   $\varphi_i$ for any $i\in\{1,\dots,k\}$ and there exists $i_0\in\{1,\dots,k\}$ such that $p^m$ is the maximal power of $p$ which divides $\varphi_{i_0}$.   Let   $p^n$ be the maximal power of $p$ which divides  $\beta$. Using (\ref{incl}) we get for every $v\in {\cal J}$ a relation of the type
$$v\left(\begin{matrix} \varphi_1\\ \vdots \\ \varphi_k\end{matrix}\right)=\beta \left(\begin{matrix} \alpha_1\\ \vdots \\ \alpha_k\end{matrix}\right)
$$
with $\alpha_i\in A$. The $i_0$-th equality shows that whenever  $n>m$ we must have $p^{n-m}| v$, hence the inclusion (\ref{Jideal}) is proved. We can complete now the chain of inclusions (\ref{incl}) to 

\begin{equation}\label{inclnew}\Ann(N){\cal J}\subset \beta A\subset {\cal J}\subset (\beta/(a,\beta)) A\ .
\end{equation}

Taking quotients with respect to the ideal $\beta A$ of the last two $A$-modules we get
$$f(N)\subset  \qmod{(\beta/(a,\beta))A}{\beta A }=m_{(a,\beta)\beta}\left(\qmod{A}{(a,\beta)A}\right)\ .
$$
Since $m_{(a,\beta)\beta}$ is a monomorphism this proves that $f$ can be written as $m_{(a,\beta)\beta}\circ g$ for a morphism $g\in \Hom(N,\qmod{A}{(a,\beta) A})$, which completes the proof.
\end{proof}

\begin{pr} \label{gcd} Suppose that $A$ is a unique factorization domain, $N$ is a finitely generated torsion module, and $\Ann(N)$ is finitely generated. Put $a:=\gcd(\Ann(N))$. 
\begin{enumerate}[(i)]
\item The natural morphism
$$c:\Hom\left(N,\qmod{A}{a A}\right)\to \varinjlim_{\alpha\in A^*}\Hom\left(N,\qmod{A}{\alpha A}\right)
$$
is an isomorphism.
\item One has  a canonical isomorphism 
$$\Hom\left(N,\qmod{A}{a A}\right)\textmap{\simeq}\Ext^1(N,A)\ .
$$
\end{enumerate}

\end{pr}

\begin{proof}

Since the morphisms $\mu_{\alpha\beta}$ which intervene in the inductive limit    are injective, it follows that $c$ is injective, so it suffices to prove that $c$ is surjective.

Let $\beta\in A^*$ and $f\in \Hom\left(N,\qmod{A}{\beta A}\right)$. By Lemma \ref{hope} there exists %
$$g\in  \Hom(N,\qmod{A}{(a,\beta) A})$$
such that $f=\mu_{(a,\beta)\beta}(g)$. Noting that $(a,\beta)|a$, we see that the three elements 
$$g\in  \Hom(N,\qmod{A}{(a,\beta) A})\ ,\ \mu_{(a,\beta)a}(g)\in \Hom\left(N,\qmod{A}{a A}\right)\ ,\ f=\mu_{(a,\beta)\beta}(g) 
$$
define the same element in the inductive limit. Therefore $\mu_{(a,\beta)a}(g)$ is a pre-image (with respect to $c$) of the element defined by $f$ in the inductive limit.

\end{proof}

 Corollary 4.11 p. 133 in \cite{BS} states:
\begin{pr}\label{BS}
Let $\pi:M\to B$ be a proper morphism of complex spaces  and ${\cal E}$ a coherent sheaf on $M$ which is flat over $B$. Then there exists a coherent sheaf ${\cal T}$ on $B$ (unique up to isomorphism) with the following property: for every coherent sheaf ${\cal A}$ on an open set $U\subset B$ there exists an isomorphism
$$(\pi_U)_*({\cal E}_U\otimes \pi_U^*({\cal A}))\textmap{\simeq} {\cal H}om({\cal T}_U,{\cal A})
$$
which is functorial with respect to ${\cal A}$.
\end{pr}

In this statement ${\cal E}_U$, $\pi_U$ denote the restrictions of ${\cal E}$ and $\pi$ to $\pi^{-1}(U)$, and ${\cal T}_U$ the restriction of ${\cal T}$ to $U$.\\

We will see that the torsion sheaf $\Tors(R^1\pi_*({\cal E}))$ in which we are interested has a simple description in terms of the associated sheaf ${\cal T}$. Before stating our result we recall that for any complex manifold $X$ and coherent  ideal sheaf ${\cal I}\subset {\cal O}_X$ (which is non-trivial on any connected component) there exists a minimal locally principal ideal sheaf $\gcd({\cal I})$ containing ${\cal I}$. At a point $x\in X$ the stalk $\gcd({\cal I})_x$ is generated by $a:=\gcd(\varphi^1,\dots,\varphi^k)$, where $(\varphi^1,\dots,\varphi^k)$ is a system of generators of the stalk ${\cal I}_x$.  The definition is coherent because, choosing holomorphic functions $\tilde a$, $\tilde\varphi^i$ representing the germs $a$, $\varphi^i$, we have $\tilde a_u=\gcd(\tilde\varphi^1_u,\dots, \tilde \varphi^k_u)$ for every $u\in X$ which is sufficiently close to $x$ (see Theorem 2.11 \cite{De} p. 93).

The   effective divisor $D_{\max}(Z)$ corresponding to $\gcd({\cal I})$ is the maximal effective divisor contained in the complex space $Z$ defined by ${\cal I}$. This maximal divisor might be of course empty (when $\gcd({\cal I})={\cal O}_X$).  

I am grateful to Daniel Barlet for explaining me interesting geometric constructions of $D_{\max}(Z)$. For instance,  this divisor can be obtained using the  blow up $p:\hat X\to X$ of $X$ along $Z$ and defining $D_{\max}(Z)$  as image (as analytic cycle) of the divisor on $\hat X$ corresponding to the locally principal sheaf of ideals $p^*({\cal I})/\Tors\subset {\cal O}_{\hat X}$.

\begin{thry}\label{new-desc}
Let $\pi:M\to B$ be a proper morphism of complex spaces  and ${\cal E}$ a coherent sheaf on $M$ which is flat over $B$. Suppose that $B$ is locally irreducible, and let ${\cal T}$ be the associated coherent sheaf on $B$ given by Proposition \ref{BS}. Then
\begin{enumerate}[(i)]
\item One has canonical isomorphisms
$${\cal E}xt^1({\cal T},{\cal O}_B)\textmap{\simeq}\Tors(R^1\pi_*({\cal E}))\ .
$$
\item If $\pi_*({\cal E})=0$ and $B$ is smooth, then
\begin{enumerate}
\item ${\cal T}$ is a torsion sheaf,
\item Denoting by $D_{\max}$ the maximal divisor contained in the complex subspace defined by   $\Ann({\cal T})$, one has canonical isomorphisms
$${\cal H}om({\cal T}_{D_{\max}},{\cal O}_{D_{\max}})\textmap{\simeq}{\cal H}om({\cal T},{\cal O}_{D_{\max}})\textmap{\simeq}\Ext^1({\cal T},{\cal O}_B)\textmap{\simeq}\Tors(R^1\pi_*({\cal E}))\ .
$$
\end{enumerate} 
 
\end{enumerate}

\end{thry}

\begin{proof}

(i) Let $U\subset B$ be a connected open set and $\varphi\in{\cal O}(U)\setminus\{0\}$. We put $\Phi:=\varphi\circ \pi$, and we denote by $D_\varphi\subset U$, $D_\Phi\subset \pi^{-1}(U)$ the effective divisiors defined by $\varphi$ and $\Phi$ respectively.   
Since ${\cal E}$ is flat over $B$,  the exact sequence
$$0\to {\cal O}_U\textmap{ \varphi\cdot }{\cal O}_U\to {\cal O}_{D_\varphi}\to 0
$$
on $B$ gives  an exact sequence
$$0\to {\cal E}_{ U}\textmap{\Phi\cdot } {\cal E}_{ U}\to {\cal E}_{D_\Phi}\to 0\  
$$
of coherent sheaves on $M$ (see \cite{BS} p. 172). The corresponding long exact sequence on $U$ is
$$0\to (\pi_U)_*({\cal E}_U)\textmap{ \varphi\cdot } (\pi_U)_*({\cal E}_U)\to (\pi_U)_*({\cal E}_{D_\Phi})\to R^1(\pi_U)_*({\cal E}_U)\textmap{m_\varphi}R^1(\pi_U)_*({\cal E}_U)\ ,
$$
which gives a canonical isomorphism
\begin{equation}\label{kermphi} \Qmod{(\pi_U)_*({\cal E}_{D_\Phi})}{\qmod{(\pi_U)_*({\cal E}_U)}{\varphi(\pi_U)_*({\cal E}_U)}}\textmap{\simeq I_\varphi}\ker(m_\varphi)\ ,
\end{equation}
which composed with the identifications   

$$(\pi_U)_*({\cal E}_U)={\cal H}om({\cal T}_U,{\cal O}_U)\ ,\ (\pi_U)_*({\cal E}_{D_\Phi})={\cal H}om({\cal T}_U,{\cal O}_{D_\varphi})\ , $$
given by Proposition \ref{BS} gives a canonical isomorphism
\begin{equation}\label{kermphinew} \Qmod{{\cal H}om({\cal T}_U,{\cal O}_{D_\varphi})}{\qmod{{\cal H}om({\cal T}_U,{\cal O}_U)}{\varphi{\cal H}om({\cal T}_U,{\cal O}_U)}}\textmap{\simeq J_\varphi}\ker(m_\varphi)\ .
\end{equation}

For a product $\psi=\varphi\xi$ we get a morphism of short exact sequences
$$
\begin{diagram}[s=7mm,w=8mm,midshaft]
0&\rTo & {\cal E}_{\pi^{-1}(U)} & \rTo^{\Phi\cdot} & {\cal E}_{\pi^{-1}(U)} & \rTo & {\cal E}_{D_\Phi}&\rTo &0\\
&&\dTo<{\id} & & \dTo<{\Xi\cdot}& &\dTo<{\Xi\cdot}\\
0&\rTo & {\cal E}_{\pi^{-1}(U)} & \rTo^{\Psi\cdot} & {\cal E}_{\pi^{-1}(U)} & \rTo & {\cal E}_{D_\Psi}&\rTo &0
\end{diagram}\ ,
$$
 which gives the commutative diagram
$$
\begin{diagram}[s=7mm,w=10mm,midshaft]
 (\pi_U)_*({\cal E}_{D_\Phi}) & \rTo & R^1(\pi_U)_*({\cal E}_U) & \rTo^{m_\varphi} &  R^1(\pi_U)_*({\cal E}_U) \\
 \dTo<{\xi\cdot } & & \dTo<{\id}& &\dTo<{m_\xi}\\
 (\pi_U)_*({\cal E}_{D_\Psi}) & \rTo & R^1(\pi_U)_*({\cal E}_U) & \rTo^{m_\psi} & R^1(\pi_U)_*({\cal E}_U) \end{diagram}\ .
$$
This shows that, via the isomorphisms $I_\varphi$, $I_\psi$ (and also $J_\varphi$, $J_\psi$), the obvious inclusion  $\ker(m_\varphi)\subset\ker(m_\psi)$ corresponds to the morphism 
induced by multiplication with $\xi$.

For a fixed point $x\in B$ the same results will hold if we use stalks at $x$ instead of sheaves on $U$ and germs $\varphi$, $\psi$, $\xi\in{\cal O}_{B,x}^*$ instead of non-zero holomorphic functions. Therefore we obtain a canonical isomorphism

$$\varinjlim_{\varphi\in {\cal O}_{B,x}^*}\Qmod{{\cal H}om({\cal T}_x,{\cal O}_{D_\varphi,x})}{\qmod{{\cal H}om({\cal T}_x,{\cal O}_x)}{\varphi{\cal H}om({\cal T}_x,{\cal O}_x)}}\textmap{\simeq J^x} \Tors(R^1\pi_*({\cal E}))_x\ .
$$
The claim follows now by Proposition \ref{Ext1NA} (iii).
\\ \\
(ii) Suppose now that $\pi_*({\cal E})=0$. By Proposition \ref{BS} this gives ${\cal H}om({\cal T},{\cal O}_B)=0$, hence ${\cal T}$ is a sheaf of rank 0, in other words it is a torsion sheaf.  The claim follows by (i) and Proposition \ref{gcd}

\end{proof}

\section{Applications and examples} \label{App}

\subsection{Parametrizing moduli stacks of complex manifolds}

\def\Diff{\mathrm{Diff}}
\def\reg{\mathrm{reg}}
Let $S$ be a fixed compact, connected, oriented smooth   manifold,  and let ${\cal M}_S$ be the moduli set of biholomorphism classes of complex manifolds which are diffeomorphic to $S$ . This moduli set has a natural topology obtained by   identifying it with the quotient ${\cal J}/\Diff(M)$ of the space ${\cal J}$ of integrable complex structures on   $S$ by the group $\Diff(M)$ of diffeomorphisms of $S$.  It is well known that, even in simple cases, this quotient topology is highly non-Hausdorff. 

%

 In general the moduli space ${\cal M}_S$ is not a complex space, but can be naturally regarded as a complex analytic stack, i.e. a ``category fibered in groupoids" over the  category ${\cal C}$ of complex spaces (endowed with the usual Grothendieck topology) satisfying the two standard  conditions: the isomorphisms form a sheaf and every descent datum is effective (see Definition 4.3 \cite{Fa}). The ``fiber" of the   stack ${\cal M}_S$ over a complex space $B$  is just the groupoid of flat holomorphic families ${\cal X}\to B$ whose fibers are complex manifolds diffeomorphic to $S$. Intuitively, such a family can be interpreted as a holomorphic morphism $B\to {\cal M}_S$.  
A detailed construction of the moduli stack  ${\cal M}_S$ can be found in \cite{Me}.

We will concentrate on the open moduli substack ${\cal M}^{\reg}_S\subset{\cal M}_S$ of complex manifolds $X$ diffeomorphic to $S$ and having $H^2(\Theta_X)=0$. The versal deformation of a manifold $X$   with $H^2(\Theta_X)=0$ is smooth, so the complex analytic stack ${\cal M}^{\reg}_S$ is locally isomorphic to the  quotient  stack of a complex manifold  by a groupoid complex space (see \cite{Ho}  sect. 3 for the analogous concept in the algebraic geometric framework). 
An interesting example of such an analytic stack (obtained as quotient of a smooth complex manifold) is described below; it can be interpreted as the moduli stack classifying a family of contracting germs, so a family of minimal class VII surfaces (see \cite{OT} p. 341, the example at the end of sect. 7):
\\ \\
{\bf Example:} Consider the action of $\alpha:\C\times \C^2\to\C^2$ given by 
$$\alpha(\zeta,(z_1,z_2)):=(z_1,z_2+z_1\zeta)\ .
$$

The map $p:\C^2/\alpha\to\C$ given by $[(z_1,z_2)]\to z_1$ is surjective, its  fiber over any point  $z\in\C^*$  is a point, whereas the fiber over 0 is the  line $\{[(0,u)]|\ u\in\C\}$. The  topological quotient $\C^2/\alpha$ can be intuitively thought of as a complex line  with a 1-parameter family  of  mutually non-separable ``origins" $0_u:=[(0,u)]$. Note that the  dimension of the formal Zariski tangent space of the quotient stack $\C^2/\alpha$ at $[(z_1,z_2)]$ is 1  for $z_1\ne 0$ and  2 for $z_1=0$. Therefore, as happens with many moduli stacks of the form ${\cal M}^{\reg}_S$, the dimension of the Zariski tangent spaces is non-constant (it ``jumps"). 
Consider the family of holomorphic maps $f_u:\C\to\C^2$  given by $f_u(z)=(z,u)$ and note that the images of the induced maps $g_u:\C\to \C^2/\alpha$ cover the quotient $\C^2/\alpha$.  Any $g_u$ is injective on whole $\C$ and étale on $\C^*$. The stack $\C^2/\alpha$ can be reconstructed using this family of ``parametrizations" and studying how their images   glue  together in the stack. Indeed, in this way we see that $\C^2/\alpha$  can be obtained from a 1-parameter family $(\C_u)_{u\in\C}$ of copies of $\C$ by identifying all punctured lines $\C^*_u$ to $\C^*$ in the obvious way.    This example shows that for understanding geometrically a moduli stack ${\cal M}^{\reg}_S$ is important to study the following
\\ \\
{\bf Problem:} {\it Classify generically versal holomorphic families $\pi:{\cal X}\to B$ with smooth base $B$,  whose fibers $X_b$ are  connected complex manifolds diffeomorphic to $S$ and  have $H^2(\Theta_{X_b})=0$. For such a family describe explicitly  the analytic  subset  $B_{\mathrm{nv}}\subset B$ of points where $\pi$ is non-versal.   
}
\\

The notion of ``versality" used here follows the terminology introduced in \cite{BHPV}  section I.10, hence it corresponds to ``semi-universality"  in the terminology of \cite{Dou}, \cite{Bi} and \cite{Fl}. 

In order to explain in detail the condition {\it generically versal} used above, consider  
 a proper holomorphic submersion (with connected base and connected fibers) $\pi:{\cal X}\to B$, denote by $\Theta^v_{\cal X}$ the vertical tangent subbundle of the tangent bundle $\Theta_{\cal X}$.  As usually we identify holomorphic vector bundles with the associated locally free coherent sheaves. The long exact sequence associated with the short exact sequence 
$$0\to \Theta^v_{\cal X}\to \Theta_{\cal X}\textmap{\pi_*} \pi^*(\Theta_B)\to 0 \ ,
$$
of sheaves on ${\cal X}$ reads
$$\dots\to R^0\pi_*\pi^*(\Theta_B)=\Theta_B\textmap{\delta} R^1\pi_*(\Theta^v_{\cal X}) \to R^1\pi_*(\Theta _{\cal X})\to\dots \ .
$$
We denote by 
$$\delta_b: (\Theta_B)_b\to R^1\pi_*(\Theta^v_{\cal X})_b\ ,\ \delta(b):\Theta_B(b)\to R^1\pi_*(\Theta^v_{\cal X})(b)$$
the morphisms induced by $\delta$ between the stalks, respectively the fibers of the two sheaves at a point $b\in B$. 
 The Kodaira-Spencer map $\rho_{b}:\Theta_B(b)\to   H^1(X_b,\Theta_b)$ can be written as 
 $$\rho_{b}= c_b\circ\delta(b)\ ,
 $$ 
 where $c_b$ denotes the canonical map $R^1\pi_*(\Theta^v_{\cal X})(b)\to H^1(X_b,\Theta_{X_b})$ (see \cite{BS} p. 112).

 \begin{pr}\label{EqCond}
Let $\pi:{\cal X}\to B$ be a proper holomorphic submersion with connected base and connected fibers. Suppose that there exists $b_0\in B $ such that $\rho_{b_0}$ is an isomorphism and the  map $b\mapsto h^1(\Theta_{X_b})$ is constant on an open neighborhood $U$ of $b_0$.  For a point $b\in B$ the following conditions are equivalent: 
\begin{enumerate}
\item $\rho_b$ is bijective,
\item $\rho_b$ is surjective,
\item $h^1(\Theta_{X_b})=\dim(B)$ and $\delta_b$ is an isomorphism of ${\cal O}_{B,b}$-modules.
\item $h^1(\Theta_{X_b})=\dim(B)$ and $\delta_b$ is an epimorphism of ${\cal O}_{B,b}$-modules.

\end{enumerate}
\end{pr}

\begin{proof}

(1)$\Rightarrow$(2) is obvious.  
Put $n:=\dim(B)$.  Using the hypothesis and the semicontinuity theorem \cite{BS}, \cite{Fl}, \cite{DV}, it follows that $n$ is the minimal value of this map on $B$, hence the set 
$$B_0:=\{b\in B|\ h^1(\Theta_{X_b})=n\}$$
 (which obviously contains $U$) is a non-empty Zariski open set.  
Suppose that (2) holds. The surjectivity of $\rho_b$ implies $h^1(\Theta_{X_b})\leq n$, hence in fact $h^1(\Theta_{X_b})=n$. Therefore  $b\in B_0$ and  $\rho_b$ is bijective. On the other hand $\Theta^v_{\cal X}$ is cohomologically flat in dimension 1 over $B_0$, hence  $R^1\pi_*(\Theta^v_{\cal X})$ is locally trivial of rank $n$ on $B_0$ and the canonical map $c_b$ is an isomorphism for every $b\in B_0$ (see \cite{BS} p. 134). This shows that, for $b\in B_0$, the bijectivity of $\rho_b$ is equivalent to the bijectivity of $\delta(b)$. On the other hand, for $b\in B_0$ the ${\cal O}_{B,b}$-modules  $(\Theta_B)_b$, $R^1\pi_*(\Theta^v_{\cal X})_b$ are free of rank $n$, hence
$$ \delta(b) \hbox{ is bijective } \Leftrightarrow \det(\delta(b))\ne 0\Leftrightarrow \det(\delta_b)\in {\cal O}_{B,b}^*\Leftrightarrow \delta_b \hbox{ is invertible}\ .$$
Therefore (2)$\Rightarrow$(3).  Obviously (3)$\Rightarrow$(4). On the other hand 
$$\delta_b\hbox{ is surjective } \Rightarrow   \delta(b) \hbox{ is surjective },$$
which, for $b\in B_0$, implies  $\delta(b) \hbox{ is bijective}$. As above this implies that $\delta_b$ is invertible. This shows (4)$\Rightarrow$(3). Finally, if $(3)$ holds then $\delta(b)$ is an isomorphism which, for $b\in B_0$,  implies that $\rho_b=c_b\circ\delta(b)$  is an isomorphism.
  \end{proof}
  
  A holomorphic family which satisfies the assumptions of Proposition \ref{EqCond} will be called generically versal. The terminology is justified by the following corollary, which can be regarded as a special case of the classical Zariski openness of versality  \cite{Bi}, \cite{Dou}, \cite{Fl}. Note however that in the terminology of these authors the  versality condition  is weaker than ours, so the corollary below is not  formally a particular case of   this classical result. 
\begin{co}
Let $\pi:{\cal X}\to B$ be a proper holomorphic submersion with connected base and connected fibers, such  that there exists $b_0\in B $ for which the Kodaira-Spencer map $\rho_{b_0}$ is an isomorphism, and the  map $b\mapsto h^1(\Theta_{X_b})$ is constant on an open neighborhood $U$ of $b_0$. Then the set  
$$B_{\mathrm{v}}:=\{b\in B|\ \pi\hbox{ is versal at } b\}$$
is Zariski open in $B$.
\end{co}
\begin{proof}
First, using the definition of versality  and Kodaira-Spencer completeness theorem \cite{KS} (which applies because $B$ is smooth) we see that $\pi$ is versal at $b$ if and only if $\rho_b$ is an isomorphism.  With the notations introduced in  the proof of Proposition \ref{EqCond} and using this remark we obtain
\begin{equation}\label{Bv}
B_{\mathrm{v}}=B_0 \setminus\supp\left(\qmod{R^1\pi_*(\Theta^v_{\cal X})}{\delta(\Theta_B)}\right) \ ,
\end{equation}
which proves that $B_{\mathrm{v}}$ is Zariski open.

\end{proof}

Using Proposition \ref{forGeorges} we obtain interesting general properties of   versal holomorphic families.

\begin{pr}\label{divnew}
Let  $\pi:{\cal X}\to B$ be  a generically versal holomorphic submersion with connected surfaces as fibers such that the map $B \ni x\mapsto h^2(\Theta_{X_x})$ is constant on $B$.
 Then the analytic subset $B_{\mathrm{nv}}\subset B$ of points where $\pi$ is non-versal is an effective divisor containing $\BN_\pi$.

\end{pr}
\begin{proof}  

Using  Riemann-Roch theorem fiberwise and the hypothesis, we get $B_0=B\setminus\BN_\pi$. Taking into account formula (\ref{Bv}) it follows $\BN_\pi\subset B_{\mathrm{nv}}$. It remains to prove that  $B_{\mathrm{nv}}$ is pure 1-codimensional at any point. 

If not, there would exist a {\it regular} point $x\in B_{\mathrm{nv}}$   such that $\mathrm{codim}_x(B_{\mathrm{nv}})\geq 2$. Then  $x\in \BN_\pi$, because the intersection $B_{\mathrm{nv}}\cap (B\setminus\BN_\pi)$ is obviously an effective  divisor in  the open  set $B\setminus\BN_\pi$: it coincides with the divisor of points where the Kodaira-Spencer map (which on this set is just a morphism of holomorphic vector bundles of the same rank)   is degenerate. 

Since we supposed $\mathrm{codim}_x(B_{\mathrm{nv}})\geq 2$ and we have proved $\BN_\pi\subset B_{\mathrm{nv}}$, we obtain $\mathrm{codim}_x(\BN_\pi)\geq 2$. Choose a local chart $h:U\textmap{\simeq} V\subset \C^n$ around $x$. Applying Proposition \ref{forGeorges} to the vertical tangent bundle $\Theta^v_{{\cal X}_U}$ and the sheaf morphism
$$\delta:\Theta_U={\cal O}_U^{\oplus n}\to R^1(\pi_U)_*(\Theta^v_{{\cal X}_U}) 
$$
(where $\pi_U:{\cal X}_U\to U$ stands for the restriction of $\pi$ to ${\cal X}_U:=\pi^{-1}(U)$), it follows that $B_{\mathrm{nv}}$ is a divisor at $x$, which contradicts our assumption.
\end{proof}

\begin{co} \label{non-inj} With the notations and conditions of Proposition \ref{divnew}:
 \begin{enumerate}
 \item For a point $x\in BN_\pi$ with $\mathrm{codim}_x(\BN_\pi)\geq 2$, the intersection %
 $$U\cap(B_{\mathrm{nv}}\setminus BN_\pi)$$
  is non-empty, for every open neighborhood $U$ of $x$.
 \item The map $B\to {\cal M}_S$ induced by $\pi$ is non-injective around any point   
 $$x\in B_{\mathrm{nv}}\setminus BN_\pi\ .$$
  More precisely, let $x_0\in B_{\mathrm{nv}}\setminus BN_\pi$, $X_0$ the corresponding fiber, $\ug_0$ its universal deformation (which exists because $h^0(\Theta_{X_0})=0$), and    %
$$f_\pi^{x_0}:(B,x_0)\to \ug_{0}$$
the morphism of germs induced by $\pi$. Then either the fiber of $f_\pi^{x_0}$ has strictly positive dimension, or $f_\pi^{x_0}$ is a finite morphism  of degree $\geq 2$. 
 \end{enumerate}
\end{co}
\begin{proof}
The first statement is an obvious consequence of Proposition \ref{divnew}. For the second  note first that $B_{\mathrm{nv}}\setminus BN_\pi$ is just the subset of points $x\in B\setminus BN_\pi$ where the differential of $f_\pi^{x_0}$ is not bijective. The fiber of  $f_\pi^{x_0}$  is a germ $(F,x_0)$ of a complex space, namely the germ in $x_0$ of the complex space defined (in  a neighborhood of $x_0$) by the equation $f^{x_0}_\pi(x)=f^{x_0}_\pi(x_0)$.  If $\dim_{x_0}(F)=0$ then $x_0$ is  an isolated point in the fiber $F$, hence $f^{x_0}_\pi$ is a finite morphism at $x_0$ (see \cite{Fi} Lemma 3.2 p. 132) and its local degree   at $x_0$ is defined. This local degree coincides with $\dim_\C({\cal O}_{F,x_0})$ (\cite{EL} p. 24). Since the differential of $f^{x_0}_\pi$ is degenerate, $F$ cannot be   reduced   at $x_0$, hence $\dim_\C({\cal O}_{F,x_0})\geq 2$.
\end{proof}

Corollary \ref{non-inj} plays an important role in \cite{D2}. In this article the author constructs explicitly generically versal families of class VII surfaces parameterized by open sets $B\subset\C^n$.  Applying Corollary \ref{non-inj} one can prove that the induced  map $B\to {\cal M}$ (in the moduli space of class VII surfaces with fixed $b_2$)  is non-injective, more precisely is non-injective on every open subset which intersects the set  $B_{\mathrm{nv}}\setminus \BN_\pi$ (which is non-empty for these particular families). This conclusion is very interesting, because for two different  points $x\ne y$ of $B$ the corresponding fibers $X_x$, $X_y$ ``look" different, hence there is no obvious isomorphism $X_x\simeq X_y$. It turned out that for Dloussky's families of class VII surfaces, it's very difficult to determine explicitly the divisor $B_{\mathrm{nv}}\subset B$ and the locus of pairs $(x,y)\in B\times B$  for which $X_x\simeq X_y$.

\subsection{An explicit computation of $H^0(\Tors(R^1p_*({\cal E})))$}

Let $X$ be a minimal class VII surface with $b:=b_2(X)>0$.  Recall from \cite{Te4} that the intersection form 
$$q_X:\qmod{H^2(X,\Z)}{\Tors}\times \qmod{H^2(X,\Z)}{\Tors}\to\Z
$$
is negative definite and the lattice $H^2(X,\Z)/\Tors$ admits a basis $(e_1,\dots,e_b)$, unique up to order and called the Donaldson basis of  $H^2(X,\Z)/\Tors$, such that
 $$e_i\cdot e_j=-\delta_{ij}\ ,\ c_1({\cal K}_X)=\sum_{i=1}^b e_i\ .
 $$
 This existence theorem is obtained using Donaldson's first theorem and the fact that $c_1({\cal K}_X)$ is a characteristic element for the intersection form on $H^2(X,\Z)/\Tors$.
For a subset $I\subset\{1,\dots,b\}$ we put
$$e_I:=\sum_{i\in I} e_i\ ,\ \bar I:=\{1,\dots,b\}\setminus I\ .
$$

  We recall that, by definition, a cycle of curves in a surface $X$ is an effective divisor $D\subset X$ which is either an elliptic curve, or a rational curve with a simple singularity, or a sum $D=\sum_{s=1}^k D_s$ of $k\geq 2$ smooth rational curves $D_s$ such that 
 $$D_1D_2=\dots =D_{k-1} D_k=D_k D_1=1\ .$$
The GSS conjecture (which, if true, would complete the classification of class VII surfaces) states that any minimal class VII surface $X$ with $b_2(X)>0$ is a Kato surface (and has a global spherical shell). Any Kato surface has a cycle, so conjecturally any minimal class VII surface $X$ with $b_2(X)>0$ should have a cycle. A generic Kato surface (more precisely a Kato surface  with non-trivial trace \cite{D1}) has a homologically trivial cycle of rational curves. Such a surface is called an Enoki surface and can be obtained as a compactification of an affine line bundle over an elliptic curve. Any minimal class VII surface which has a non-zero numerically trivial divisor (effective or not) is an Enoki surface \cite{E}. 
 
There are two types of Enoki surfaces. An Enoki surface $X$ of generic type has $b$ irreducible curves  $D_1,\dots,D_b$, and these curves form a cycle of rational curves. More precisely   supposing $b>1$ each $D_i$ is smooth rational and,  with respect to a suitably ordered Donaldson  basis $(e_1,\dots,e_b)$, we have $c_1({\cal O}(D_i))=e_i-e_{i+1}$, the indices being considered modulo $b$.
 An Enoki surface of special type (also called a parabolic Inoue surface) has $b+1$ irreducible curves, namely   $b$ rational curves  forming a cycle of the same type as in the generic case, and a smooth elliptic curve $E$ with $c_1(E)=-\sum_i e_i$. \\

In our previous articles   \cite{Te2}--\cite{Te5} we developed a program  for proving the existence of curves on minimal class VII surfaces with $b_2>0$ using certain moduli spaces of polystable bundles on these surfaces: for a class VII surface $X$ we consider a differentiable rank 2 bundle $E$ on $X$  with $c_2=0$ and $\det(E)=K_X$, where $K_X$ denotes the underlying differentiable line bundle of the canonical line bundle ${\cal K}_X$. The fundamental object intervening 
in our program is the moduli space  
$${\cal M}^{\mathrm{pst}}:={\cal M}^{\mathrm{pst}}_{{\cal K}}(E)$$
 of polystable holomorphic structures on $E$ which induce the canonical holomorphic structure ${\cal K}_X$ on $K_X$. This moduli space is alway compact  but, for $b_2(X)>0$, it is not a complex space. It can be written as the union
$${\cal M}^{\mathrm{pst}}= {\cal R}\cup{\cal M}^{\mathrm{st}}\ ,
$$
where ${\cal R}$ is a finite union  of ``circles of reductions" (i.e., of split polystable bundles) and ${\cal M}^{\mathrm{st}}$ is the moduli space of stable  holomorphic structures on $E$ which induce the canonical holomorphic structure ${\cal K}_X$ on $K_X$. The latter subspace is open in  ${\cal M}^{\mathrm{pst}}$ and has a natural complex space structure.

An important role in the proofs of the main results in \cite{Te2}, \cite{Te4} (which concern the case $b_2(X)\in\{1,2\}$) and in the   program  developed for the general case, is played by reduced, irreducible, compact  complex subspaces $Y\subset {\cal M}^{\mathrm{st}}$   which contain the point corresponding to the isomorphy class $[{\cal A}]$ of the ``canonical extension" ${\cal A}$ of $X$ \cite{Te3}, \cite{Te5}. By definition ${\cal A}$ is the (essentially unique) non-trivial extension of the form 
\begin{equation}\label{A}
0\to {\cal K}_X\to {\cal A}\to {\cal O}_X\to 0\ . 
\end{equation}

This bundle is stable with respect to a suitable Gauduchon metric, unless $X$ belongs to  a very special class of Kato surfaces \cite{Te3}.

Supposing that  a classifying bundle for the embedding $Y\hookrightarrow  {\cal M}^{\mathrm{st}}$ exists, and choosing    a desingularization $B$ of $Y_{\rm red}$, we obtain:
\begin{itemize}
\item a compact complex manifold $B$,
\item a rank 2 bundle ${\cal E}$ on $B\times X$,
\item a point $a\in B$ such that the bundle ${\cal E}_a$ on $X$ is isomorphic to the canonical extension ${\cal A}$.
\end{itemize}

We denote by $\pi$, $p$ the projections of $B\times X$ on $B$ and $X$ respectively. For any  bundle ${\cal E}$ on $B\times X$ and for every line bundle ${\cal T}$ on $X$ we denote by $\TE$ the tensor product 
$$\TE:={\cal E}\otimes p^*({\cal T})\ .$$
 The goal of this section is  the following vanishing result:
 
  \begin{pr}
 Let $X$ be a minimal class VII surface with $b=b_2(X)>0$ and $J\subset\{1,\dots,b\}$ with $0<|J|<b$. Let ${\cal E}$  be a rank 2 bundle on $B\times X$, where $B$ is a connected  compact complex manifold such that
 \begin{enumerate}[(a)]
 \item $p_*({\cal E})=0$,
 \item There exists a point $a\in B$ such that ${\cal E}_a\simeq {\cal A}$.
 \end{enumerate}

Then 
\begin{enumerate}[(i)]
\item For any $[{\cal T}]\in\Pic^{-e_J}(X)$ except for at most one point of $\Pic^{-e_J}(X)$ one has  
$$ H^0(X, \Tors(R^1p_*(\TE)))=0\ ,$$
\item For any point $[{\cal T}]\in  \Pic^{-e_J}(X)$ with sufficiently negative degree\footnote{Recall that for a Gauduchon metric $g$ on a complex surface $X$ with $b_1(X)$ odd the degree map $\deg_g:\Pic(X)\to \R$ associated with $g$ is surjective on any connected component $\Pic^c(X)$ (\cite{LT} Prop. 1.3.13 p. 40). } one has
$H^1(\TE)=0$.

\end{enumerate}

\end{pr}

Note that the condition $p_*({\cal E})=0$ holds automatically when $a(B)=0$ and the bundle ${\cal E}_z$  on $X$ is non-filtrable for generic $z\in B$ (see section 4.3 in \cite{Te6}). In \cite{Te6} we use this vanishing result to prove a non-existence theorem for positive dimensional reduced, irreducible compact complex subspaces $Y\subset  {\cal M}^{\rm st}$ containing the point $[{\cal A}]$ and an open neighborhood $Y_0$ of this point such that $Y_0\setminus\{[{\cal A}]\}$ consists only of non-filtrable bundles. The idea of the proof is to show that the vanishing of $H^1(\TE)$ leads to a contradiction, which gives  $p_*({\cal E})\ne 0$, and using the results of section 4.3 in \cite{Te6}, we obtain $a(B)>0$. Therefore $B$ (so also $Y$) is covered by divisors, so the    result follows by induction with respect to $\dim(Y)$.
This non-existence result is a generalization (for subspaces of arbitrary dimension) of  Corollary 5.3 \cite{Te2} which deals with the case $\dim(Y)=1$ and plays a crucial role in proving the existence of a curve on class VII surfaces with $b_2=1$.
 
 \begin{proof}
 
(i) Recall that we denote by ${\cal D}iv(X)$ the set of {\it effective} divisors of $X$.
 Using Corollary \ref{main} we see that in order to prove this proposition it suffices to show that for any effective divisor $D\in  {\cal D}iv(X)$ we have
 $$H^0(\TE_{{\cal D}}({\cal D}))=0\ ,
 $$
where   ${\cal D}:=B\times D$, regarded as an effective divisor of $B\times X$.  The projection $\pi_{\cal D}:=\resto{\pi}{{\cal D}}:{\cal D}\to B$ is a flat morphism and 
$$\TE_{{\cal D}}({\cal D})={\cal E}_{\cal D}\otimes p^*({\cal T}(D))$$
 is a locally free sheaf on ${\cal D}$, so is flat over $B$.     
 Regarding $\TE_{{\cal D}}({\cal D})$  as a locally free sheaf on ${\cal D}$ and using the projection $\pi_{\cal D}:{\cal D}\to B$ we will prove that 
 \\
 \\
 {\it Claim:} Under the assumptions of the theorem for any $[{\cal T}]\in\Pic^{-e_J}(X)$ except for at most one point of $\Pic^{-e_J}(X)$ it holds 
 $$H^0 (\TE_{{\cal D}}({\cal D})_a)=0\ \forall D\in{\cal D}iv(X)\ .$$  
  
By Lemma \ref{tfapp} and Remark \ref{rem}  $H^0 (\TE_{{\cal D}}({\cal D})_a)=0$ implies $(\pi_{\cal D})_*( \TE_{{\cal D}}({\cal D}))=0$, which obviously implies $H^0(\TE_{{\cal D}}({\cal D}))=0$. Therefore our claim completes  the proof. Since $\TE_{{\cal D}}({\cal D})_a={\cal T}\otimes {\cal A}_D(D)$ our claim follows from Lemma \ref{van} below.
\\ \\
(ii) The Leray spectral sequence associated with the projection $p$ yields the exact sequence
$$0\to H^1(p_*(\TE))\to H^1(\TE)\to H^0(R^1p_*(\TE))\to H^2(p_*(\TE))\to H^2(\TE) \ .
$$
By the projection formula we have $p_*(\TE)=p_*({\cal E})\otimes {\cal T}$, which vanishes by hypothesis. Therefore the canonical morphism $H^1(\TE)\to H^0(R^1p_*(\TE))$ is an isomorphism.  By (i) we know that $H^0(\Tors(R^1p_*(\TE)))=0$ for any $[{\cal T}]\in \Pic^{-e_J}(X)$ except for at most a point. It suffices to prove that for any $[{\cal T}]\in \Pic^{-e_J}(X)$ with sufficiently negative degree  one has $H^0(\TS)=0$, where $\TS$ denotes the quotient of $R^1p_*(\TE)$ by its torsion subsheaf. The projection formula gives a canonical isomorphism $R^1p_*(\TE)=R^1p_*({\cal E})\otimes {\cal T}$ hence, since ${\cal T}$ is locally free, we get a canonical isomorphism
$$\TS={\cal S}\otimes {\cal T}\ ,
$$
where ${\cal S}$ is the quotient of $R^1p_*({\cal E})$ by its torsion subsheaf. Let ${\cal S}\hookrightarrow {\cal B}$ be the embedding of ${\cal S}$ in its bidual, which is a reflexive sheaf on a surface, hence is locally free. Fix a Hermitian metric $h$ on ${\cal B}$, a Gauduchon metric $g$ on $X$, and let $h_{{\cal T}}$ be a Hermite-Einstein Hermitian metric on ${\cal T}$. When the Einstein constant $c_{\cal T}$ of the Chern connection $A_{{\cal T},h_{\cal T}}$ is sufficiently negative, the Hermitian endomorphism $i\Lambda_g F_{h\otimes h_{{\cal T}}}$ of ${\cal B}\otimes {\cal T}$ becomes negative definite, hence $H^0({\cal B}\otimes{\cal T})=0$ by a well-known vanishing theorem (see Theorem 1.9 p. 52 in \cite{K}). It suffices to recall that the Einstein constant of $A_{{\cal T},h_{\cal T}}$ is proportional with $\deg_g({\cal T})$ (see Proposition 2.18 in \cite{LT}) and to note that $H^0({\cal B}\otimes{\cal T})=0$ implies $H^0(\TS)=0$, because $\TS$ is a subsheaf of ${\cal B}\otimes{\cal T}$.

 \end{proof}
\begin{lm}\label{van}
Let $X$ be a minimal class VII surface with $b=b_2(X)>0$ and $J\subset\{1,\dots,b\}$ with $0<|J|<b$. 
 Then for any $[{\cal T}]\in\Pic^{-e_J}(X)$ except for at most one point of $\Pic^{-e_J}(X)$   it holds
 $$H^0({\cal T}\otimes{\cal A}_D(D))=0\ \forall D\in{\cal D}iv(X) .$$

\end{lm}
  
  \begin{proof}
  
(i)     Tensorizing with ${\cal T}(D)$  the exact sequence of {\it locally free} sheaves (\ref{A}) and  restricting it to $D$  we obtain the exact sequence 
\begin{equation}\label{AD}
0\to {\cal T}\otimes {\cal K}_D(D)\to {\cal T}\otimes {\cal A}_D(D)\to {\cal T}_D(D)\to 0 \ ,
\end{equation}
hence $H^0({\cal T}\otimes {\cal A}_D(D))=0$ as soon as $H^0({\cal T}_D(D))=0$ and $H^0({\cal T}\otimes {\cal K}_D(D))=0$. We will see that the first space vanishes for every $[{\cal T}]\in\Pic^{-e_J}(X)$, whereas the second vanishes for every $[{\cal T}]\in\Pic^{-e_J}(X)$ except for at most a point of $\Pic^{-e_J}(X)$ (which is independent of $D$).
\\ \\
1.   $H^0({\cal T}_D(D))=0$ for every $[{\cal T}]\in\Pic^{-e_J}(X)$.  
\\ \\
 This follows by Lemma \ref{new} (iii), (iv) below. \\ \\
2. $H^0({\cal T}\otimes {\cal K}_D(D))=0$ for every $[{\cal T}]\in\Pic^{-e_J}(X)$ except for at most a point of $\Pic^{-e_J}(X)$ (which is independent of $D\in{\cal D}iv(X)$).
\\

The proof of 2. starts with the cohomology long exact sequence associated with the short exact sequence
$$0\to {\cal T}\otimes {\cal K}\to {\cal T}\otimes {\cal K}(D)\to {\cal T}\otimes {\cal K}_D(D)\to 0\ ,
$$
which reads
\begin{equation}\label{lexs}
0\to H^0({\cal T}\otimes {\cal K})\to H^0({\cal T}\otimes {\cal K}(D))\to H^0({\cal T}\otimes {\cal K}_D(D))\to H^1({\cal T}\otimes {\cal K})\to \dots\ .
\end{equation}
Note that $c_1({\cal T}\otimes {\cal K})=e_{\bar J}$, hence $H^0({\cal T}\otimes {\cal K})=0$ by Lemma \ref{new} (i) below. On the other hand  $h^2({\cal T}\otimes {\cal K})=h^0({\cal T}^\vee)=0$, where the first equality follows by Serre duality, and the second by \ref{new} (i) again. Using the Riemann-Roch theorem and taking into account that
 $$\chi({\cal T}\otimes {\cal K})=\frac{1}{2}e_{\bar J}(e_{\bar J}-c_1({\cal K}))=-\frac{1}{2}e_{\bar J}e_J =0\ ,
 $$
 it follows $h^1({\cal T}\otimes {\cal K})=0$. Therefore, using the exact sequence (\ref{lexs})   we obtain

\begin{equation}\label{chain}
h^0({\cal T}\otimes {\cal K}_D(D))=h^0({\cal K}\otimes {\cal T}(D))\ .
\end{equation}
Therefore it suffices to prove that for any $[{\cal T}]\in\Pic^{-e_J}(X)$ except for at most one point of $\Pic^{-e_J}(X)$ one has $h^0({\cal K}\otimes {\cal T}(D))=0$ for any $D\in{\cal D}iv(X)$. 
If   $h^0({\cal K}\otimes {\cal T}(D))>0$, there  exists an effective divisor $G\subset X$ such that 
\begin{equation}\label{KT}
{\cal K}\otimes {\cal T}\simeq {\cal O}(G-D)\ .
\end{equation}
This implies 
\begin{equation}\label{congruence}
\sum_{i\in \bar J} e_i=c_1({\cal O}(G-D))\ .
\end{equation}

Suppose first that $X$ is an Enoki surface. Taking into account the properties of Enoki surfaces explained above, we see that for any divisor $A$ (effective or not) on an Enoki surface of generic type  one has $c_1(A)c_1({\cal K})=0$, whereas for any divisor $A$ (effective or not) on an Enoki surface of special type  one has $c_1(A)c_1({\cal K})\equiv 0$ mod $b$. Taking $A=G-D$ this contradicts formula (\ref{congruence}). So for an Enoki surface one has $H^0({\cal T}\otimes{\cal K}_D(D))=0$ for every $[{\cal T}]\in\Pic^{-e_J}(X)$.  

Suppose now that $X$ is not an Enoki surface. The classes corresponding  to the irreducible curves   of $X$ are linearly independent in the free $\Z$-module $H^2(X,\Z)/\Tors$. Therefore there exists at most one divisor $A=A_+-A_-$ (where $A_\pm$ are effective divisors with no common irreducible component) such that $\sum_{i\in \bar J} e_i=c_1({\cal O}(A))$. This shows that the pairs of effective divisors $(D,G)$ for which (\ref{congruence}) holds are all given by
$$D=A_-+H\ ,\  G=A_++H \hbox{ where } H\in{\cal D}iv(X)\ .
$$
But for any such pair one has ${\cal O}(G-D)\simeq{\cal O}(A)$, so ${\cal T}$ must be isomorphic  with  ${\cal K}^{-1}(A)$ by (\ref{KT}).
\end{proof}

\begin{lm}\label{new}
Let $X$ be a minimal class VII surface with $b:=b_2(X)>0$ and $J\subset\{1,\dots,b\}$.  Then 
\begin{enumerate}[(i)]
\item For every effective divisor $D\subset X$ one has 
$$\langle c_1({\cal K}),[D]\rangle \geq 0\ .$$
\item For every effective divisor $D\subset X$ one has 
$$\langle -e_J,D\rangle + D^2\leq 0
$$
with equality if and only if $D$ is numerically trivial.
\item If $X$ is not an Enoki surface, then  for every   $[{\cal T}]\in\Pic^{-e_J}(X)$ and for every effective divisor $D\subset X$  one has $H^0({\cal T}_D(D))=0$.
\item If $X$ is an Enoki surface and $0<|J|<b$, then  for every   $[{\cal T}]\in\Pic^{-e_J}(X)$ and for every effective divisor $D\subset X$  one has $H^0({\cal T}_D(D))=0$.
\end{enumerate}
\end{lm}
\begin{proof}

(i)  This  important inequality is due to Nakamura (see \cite{Na3} Lemma 1.1.3). It can be proved easily using the fact that, since $X$ is minimal, the claimed inequality holds for any irreducible curve of $X$.
\\ \\
(ii) Decomposing $c_1({\cal O}(D))=\sum_{i=1}^b d_i e_i$ mod $\Tors$ (with $d_i\in\Z)$ we have 
$$\langle -e_J,D\rangle + D^2= \sum_{j\in J} d_j-\sum_{i=1}^b d_i^2=-\bigg(\sum_{j\in J}d_j(d_j-1)+\sum_{i\in \bar J} d_i^2\bigg)\leq 0\ ,
$$
with equality if   and only if $d_j\in\{0,1\}$ for $j\in J$ and $d_i=0$ for $i\in\bar J$. Therefore one has equality if and only if  there exists $J'\subset J$ such that $c_1({\cal O}(D))=e_{J'}$. By Lemma \ref{new} (i) we must have $J'=\emptyset$, hence $D$ is numerically trivial.\\ \\
  Let $D=\sum_{s=1}^k n_s D_s$  (with $n_s>0$) be the decomposition of $D$ in irreducible components. We will prove (iii), (iv) by induction with respect to $n:=\sum_{s=1}^k n_s$. We have
\begin{equation}\label{ctd}\sum_{s=1}^k n_s \langle c_1({\cal T}(D)),D_s\rangle=\langle c_1({\cal T}(D)),D\rangle=\langle c_1({\cal T}),D\rangle+ D^2=-\langle e_J,D\rangle + D^2\ .
\end{equation}
{\ }\\
(iii)  If $X$ is not an  Enoki surface and $D$ is not empty, then (ii) and (\ref{ctd}) show  that 
$$\sum_{s=1}^k n_s \langle c_1({\cal T}(D)),D_s\rangle<0\ ,$$
hence there exists $\sigma\in\{1,\dots,k\}$ such that $\langle c_1({\cal T}(D)),D_\sigma\rangle<0$. Writing $D=D'+D_\sigma$ we obtain a short exact sequence
$$0\to{\cal T}_{D'}(D')\to {\cal T}_D(D)\to {\cal T}_{D_\sigma}(D)\to 0 $$
which gives the long exact sequence  
$$0\to H^0({\cal T}_{D'}(D'))\to H^0({\cal T}_{D}(D))\to H^0({\cal T}_{D_\sigma}(D))\to\dots\ .
$$
But $H^0({\cal T}_{D_\sigma}(D))=0$, because the degree of the restriction of the line bundle ${\cal T}(D)$ to the irreducible curve $D_\sigma$ is negative. Therefore $H^0({\cal T}_{D}(D))=H^0({\cal T}_{D'}(D'))$, and the sum  of the coefficients of the decomposition  of $D'$ in irreducible components is $n'=n-1$.
\\
\\
(iv) If $X$ is an Enoki surface, the proof of (iii) applies except when $D$ is numerically trivial, i.e, when $D=nC$ where $C$ is the cycle of rational curves of $X$. Write $C=\sum_{i=1}^b D_i$ where (for a suitably ordered Donaldson basis $(e_i)_{1\leq i\leq b}$) $c_1({\cal O}(D_i))=e_i-e_{i+1}$, the indices being considered mod $b$. In this case we get
$$\langle c_1({\cal T}(D)),D_i\rangle =\langle -e_J,D_i\rangle+\langle D,D_i\rangle=\langle -e_J,D_i\rangle\ .
$$
Since we assumed $J\ne\emptyset$ and $J\ne\{1,\dots,b\}$ there exists $\sigma\in \bar J$ such that $\sigma+1\in J$. Then $\langle c_1({\cal T}(D)),D_\sigma\rangle=-1$, which shows that $H^0({\cal T}_{D_\sigma}(D))=0$, hence the same argument by induction as in the proof of (iii) can also be used in the case when $D$ is numerically trivial.
\end{proof}

\section{Appendix: Torsion free direct images}
It is well known that ``under suitable assumptions" the direct image of a locally free sheaf  is torsion free. Since we need this result in a very general framework (when the total space is not supposed to be reduced)   we state this result explicitly   and give a complete proof.

\begin{lm} \label{tfapp} Let $M$, $B$  be complex spaces with $B$   locally irreducible\footnote{According to the terminology of \cite{GR}, a locally irreducible complex space is automatically reduced.}, $\pi:M\to B$ a proper holomorphic map  and ${\cal E}$ a  coherent sheaf on $M$ which is flat over $B$. Then 
\begin{enumerate}
\item $\pi_*({\cal E})$ is a torsion-free coherent sheaf on $B$.
\item  Suppose  $B$ is connected and there exists $x_0\in B$ such that $H^0({\cal E}_{x_0})=0$. Then $\pi_*({\cal E})=0$.
\end{enumerate}

\end{lm}

\begin{proof}
(1) Let $x\in B$ and let $\sigma\in \pi_*({\cal E})_x$ be a torsion element of the ${\cal O}_{B,x}$-module $\pi_*({\cal E})_x$. Therefore there exists $\psi\in {\cal O}_{B,x}\setminus\{0\}$ such that $\psi\sigma=0$.  The stalk $\pi_*({\cal E})_x$ can be identified with the cohomology space $H^0(M_x,{\cal E})$ (see \cite{Go} Rem. 4.17.1 p. 202, \cite{BS} Lemma 1.3 p. 93), i.e., with the space of sections of the restriction  $\resto{{\cal E}}{M_x}$ in the sense of the theory of sheaves on topological spaces. Note that this restriction is {\it not} coherent on $M_x$. Using the natural ${\cal O}_{B}$-module structure of   ${\cal E}$, the condition $\psi\sigma=0$ becomes
\begin{equation}\label{sigma}
\psi \sigma_y=0 \hbox{ in } {\cal E}_y\  \forall y\in M_x\ .
\end{equation}

Since ${\cal O}_{B,x}$ is an integral domain by assumption (see \cite{GR} p. 8) the multiplication with $\psi$ defines a monomorphism of ${\cal O}_{B,x}$-modules 
$$0\to {\cal O}_{B,x}\textmap{\psi} {\cal O}_{B,x}\ ,
$$
hence, since ${\cal E}$ is flat over $B$, the induced morphism ${\cal E}_y\to {\cal E}_y$ is also injective for every $y\in M_x$. Therefore (\ref{sigma}) implies $\sigma=0$.

An alternative proof can be obtained using Corollary 4.11 p. 133 of \cite{BS} in the special case $\Mg={\cal O}_B$. It follows that there exists a coherent sheaf ${\cal T}$ on $B$ and an isomorphism $\pi_*({\cal E})=\Hom({\cal T},{\cal O}_B)$. It suffices to note that $\Hom({\cal T},{\cal O}_{B})$ is torsion free when the structure sheaf ${\cal O}_B$ is a sheaf of integral domains.
\\ \\
(2)  We apply Grauert's Theorem (Theorem 4.12  p. 134 [BS]).
It follows that the map $x\mapsto H^0({\cal E}_x)$ is upper semicontinuous  on $B$. Since we assumed that this function vanishes at $x_0\in B$, it follows that this function vanishes identically on a  open  neighborhood $U$ of $x_0$.  Since $B$ is reduced, the second part of the quoted theorem shows that then the sheaf $\pi_*({\cal E})$ is locally free on $U$. Therefore this sheaf vanishes on $U$ which shows that its support  is a proper analytic set of $B$  hence, since $B$ is locally irreducible,  it is a torsion sheaf (see Theorem 9.12  p. 60 \cite{Re}). On the other hand by (1) this sheaf is torsion free.  
\end{proof}

Note that in this lemma we don't have to assume  the total space or the fibers to be reduced. For instance we have
\begin{re}\label{rem}
   Suppose that ${\cal E}$ is locally free on $M$ and  $p:M\to Y$ is a flat morphism (for instance a holomorphically locally trivial morphism). Then ${\cal E}$ is flat over $Y$, hence (supposing that $Y$  is locally irreducible and $p$ is proper) Lemma \ref{tfapp} applies.
\end{re}

\end{document}